\documentclass[10pt]{amsart}
\usepackage{amssymb,amsfonts,amsmath,amsthm}
\usepackage[all,2cell]{xy}
\usepackage{graphicx}
\usepackage{pstricks}
\usepackage[mathscr]{eucal}
\usepackage{amscd}
\usepackage{color}		% Need the color package
\usepackage{epsfig}
\usepackage{enumerate}
\usepackage{indentfirst}
\usepackage{fancyhdr}

\theoremstyle{plain}
\newtheorem{thm}{\bf Theorem}[section]
\newtheorem{prop}[thm]{\bf Proposition}
\newtheorem{lem}[thm]{\bf Lemma}
\newtheorem{cor}[thm]{\bf Corollary}

\theoremstyle{definition}
\newtheorem{dfn}[thm]{\bf Definition}

\theoremstyle{remark}
\newtheorem{rem}[thm]{\bf Remark}

\DeclareMathOperator{\Hom}{Hom}

\def \A{\mathbb{A}}
\def \G{\mathbb{G}_m}
\def \F{\mathbb{F}_n}
\def \Z{\mathbb{Z}}
\def \Q{\mathbb{Q}}
\def \C{\mathbb{C}}
\def \k{\Bbbk}

\def \P{\mathbb{P}}

\newcommand{\acknowledge}{\subsection*{Acknowledgments}}

\title[Equivariant operational Chow rings]
{Equivariant operational Chow rings of 
$\mathbf{T}$-linear schemes} 
\author[Richard P. Gonzales]{Richard P. Gonzales\,*} \thanks{* Supported by the Institut des Hautes \'{E}tudes Scientifiques, T\"{U}B\.{I}TAK Project No. 112T233, and DFG Research Grant PE2165/1-1.}
\address{Institut des Hautes \'{E}tudes Scientifiques\\ 
35 Route de Chartres,\\
F-91440 Bures-sur-Yvette\\
France}
\address{
Faculty of Engineering and Natural Sciences\\
Sabanc\i \, \"{U}niversitesi\\
Orhanli, Tuzla 34956, Istanbul\\
Turkey
\vspace{.5cm}}
\email{rgonzalesv@gmail.com}

\begin{document}

\begin{abstract} 
We study $T$-linear schemes, 
a class 
of objects that includes 
spherical and Schubert varieties. 
%We provide a 
%K\"{u}nneth formula 
%for the equivariant Chow groups 
%of these schemes. 
%Using such formula, we  
%We show that equivariant Kronecker duality 
%holds for the equivariant operational Chow rings 
%(or equivariant Chow cohomology)  
%of 
%$T$-linear schemes.  
%Building on this, we obtain: (i) a 
We provide a  
localization theorem for the equivariant Chow cohomology    
of these schemes  %complete $T$-linear schemes, 
that does not depend on 
resolution of singularities. 
Furthermore, we give  
 an explicit 
presentation 
of the equivariant Chow cohomology of 
possibly 
singular complete spherical varieties    
%in those cases where  
admitting 
a smooth equivariant envelope (e.g., group embeddings).  
\end{abstract}

\maketitle

%%--------------------Here the manuscript starts------------------------------
\section{Introduction and motivation}
%\addcontentsline{toc}{section}{Introduction}
Let $\k$ be  an 
algebraically closed field. 
Let $G$ be a connected reductive linear algebraic group  
%defined over an algebraically closed field $\k$. %of characteristic zero. 
(over $\k$).
Let $B$ be a Borel subgroup of $G$ and $T\subset B$ 
be a maximal torus of $G$. 
An algebraic variety $X$, equipped with an action of $G$, 
is {\em spherical} if it contains a dense orbit of $B$.  
(Usually spherical 
varieties are assumed to be normal but this condition is not needed here.)  
Spherical varieties have been extensively studied 
in the works of Akhiezer, Brion, %\cite{bri:sph}, \cite{bri:cd}, 
Knop, %\cite{knop:sph} 
Luna, 
Pauer, Vinberg,  
Vust and others.   
For an up-to-date discussion of 
spherical varieties, as well as a comprehensive bibliography, 
see \cite{ti:sph} and the 
references therein.  %-Vust \cite{lv:sph}. 
If $X$ is spherical, then it has a finite number of $B$-orbits, 
and thus, also a finite number of $G$-orbits \cite{ti:sph}. 
In particular, $T$ acts on $X$ with a finite number of fixed points. 
These properties make spherical varieties particularly well suited for 
applying the methods of Goresky-Kottwitz-MacPherson \cite{gkm:eqc}, nowadays called GKM theory,  
in the topological setup,  
and Brion's extension of GKM theory \cite{bri:eqchow} to the
algebraic setting of equivariant Chow 
groups, as defined by Totaro, Edidin and Graham \cite{eg:eqint}. 
Through this method, substantial information about the topology and 
geometry of a spherical variety can be obtained by 
restricting one's attention to the induced action of $T$. 

\smallskip

Examples of spherical varieties include 
$G\times G$-equivariant embeddings of $G$ (e.g., toric varieties are 
spherical) and the regular symmetric varieties of 
De Concini-Procesi \cite{dp:sym}. 
The equivariant cohomology and equivariant Chow groups % \cite{eg:eqint}  
of smooth complete %projective 
spherical varieties have been studied 
by Bifet, De Concini and Procesi \cite{bif:reg}, 
De Concini-Littelmann \cite{lp:equiv}, 
Brion \cite{bri:eqchow} and Brion-Joshua \cite{bj:chern}. 
In these cases, there is a  
comparison result relating equivariant cohomology
with equivariant Chow groups: for a smooth complete %projective %%% Give a proof?
spherical variety, the equivariant cycle map yields 
an isomorphism from the equivariant Chow group to the 
equivariant (integral) cohomology (Proposition \ref{equiv.cycle.smooth.complete}).    
As for the study of the equivariant Chow groups of 
possibly singular spherical varieties, 
some progress has been made 
by %Danilov \cite{da:tor}, 
%
%Brion \cite{bri:eqchow}, 
Payne \cite{p:t}  
and the author \cite{go:equiv}-\cite{go:rm}. 
 
\smallskip

The problem of developing %(equivariant) 
intersection theory on singular varieties %schemes 
comes 
from the fact that the %equivariant 
Chow groups $A_*(-)$ do not 
admit, in general, a natural ring structure or 
intersection product. But when singularities 
are mild, for instance when $X$ is a quotient of a 
smooth variety $Y$ by a finite group $F$, 
then $A_*(X)\otimes\Q\simeq (A_*(Y)\otimes \Q)^F$, and so
$A_*(X)\otimes\Q$ inherits the ring structure of $A_*(Y)\otimes\Q$. 
To simplify notation, 
if $A$ is a $\Z$-module, %or an abelian group,   
we shall write hereafter %, throughout the paper, 
$A_\Q$ for the rational vector space $A\otimes\Q$.  
%(the tensor product is understood to be taken 
%over $\Z$).  

%This happens, for instance, in the case of simplicial toric varieties.  
%
%Throughout this paper, 
%we consider {\em rational} Chow groups 
%$A_*(X)\otimes \Q$, so we drop $\Q$ from the notation in the sequel. 

%%%%%%%%%%%%%%%%%%%%%%%%%%%%%%%%                                                                                %%%%%%%%%%%%%%%
%%%%%%%%%%%%%%%%%%%%%%%%%%%%%%%% Be careful distinguishing which results are valid with \Z rather than \Q coeff %%%%%%%%%%%%%%%
%%%%%%%%%%%%%%%%%%%%%%%%%%%%%%%%%                                                                               %%%%%%%%%%%%%%%

\smallskip

In order to study more general singular schemes, %varieties, 
Fulton and MacPherson \cite{f:int} 
introduced the notion of operational Chow groups or Chow cohomology. 
Similarly, Edidin and 
Graham defined the {\em equivariant operational Chow groups} \cite{eg:eqint}, 
which we briefly recall. (For our conventions on varieties and schemes, 
%as well as our underlying category, 
see Section 2.1.)   
Let $X$ be a $T$-scheme. 
The $i$-th $T$-equivariant operational Chow group of $X$, 
denoted $A^i_T(X)$, is defined as follows.  
An element $c\in A^i_T(X)$ is a collection 
of homomorphisms $c^{(m)}_f:A^T_m(Y)\to A^T_{m-i}(Y)$, 
written $z\mapsto f^*c\cap z$, 
for every $T$-equivariant map $f:Y\to X$ 
and all integers $m$ 
(the underlying category is the category of $T$-schemes).  
Here
$A^T_*(Y)$ denotes the %(rational) 
equivariant Chow group of $Y$ (Section 2.1). 
As in the case of ordinary operational Chow groups,  
%(\cite[Chapter 17]{f:int}), 
these homomorphisms 
%should be
%compatible with the operations on 
must satisfy three conditions of compatibility: 
with proper pushforward (resp. flat pull-back, 
resp. intersection with a Cartier divisor) 
for $T$-equivariant maps $Y'\to Y\to X$, 
with $Y'\to Y$ proper (resp. flat, resp. determined 
by intersection with a Cartier divisor); 
see \cite[Chapter 17]{f:int} for precise statements. 
The homomorphism $c^{(m)}_f$ determined by an 
element $c\in A^i_T(X)$ is 
usually denoted simply by $c$, with an indication of where it acts.  
%equivariant Chow groups (pull-back for l.c.i. morphisms, 
%proper push-forward, etc.). 
For any $X$, the ring structure on 
$A^*_T(X):=\oplus_i A^i_T(X)$ is given by composition of 
such homomorphisms. The ring $A^*_T(X)$ is graded, 
and $A^i_T(X)$ can be non-zero for any $i\geq 0$.
The most salient functorial properties of equivariant 
operational Chow groups are summarized below:
\begin{enumerate}[(i)]
 \item Cup products $A^p_T(X)\otimes A^q_T(X)\to A^{p+q}_T(X)$, $a\otimes b\mapsto a\cup b$, 
making $A^*_T(X)$ into a graded associative ring (commutative when 
resolution of singularities is known).

 \item Contravariant graded ring maps 
$f^*:A^i_T(X)\to A^i_T(Y)$ for arbitrary equivariant 
morphisms $f:Y\to X$. 

 \item Cap products $A^i_T(X)\otimes A^T_m(X)\to A^T_{m-i}(X)$, 
$c\otimes z\mapsto c\cap z$, making 
$A^T_*(X)$ into an $A^*_T(X)$-module 
and satisfying the projection formula. 

\item If $X$ is a nonsingular $n$-dimensional $T$-variety, then  
the Poincar\'e duality 
map from $A^i_T(X)$ to $A^T_{n-i}(X)$, 
taking $c$ to $c\cap [X]$, is an isomorphism, and the ring 
structure on $A^*_T(X)$ is that determined by 
intersection products of cycles on the mixed spaces $X_T$ 
\cite[Proposition 4]{eg:eqint}. 

\item Equivariant vector bundles on $X$ have 
equivariant Chern classes in $A^*_T(X)$.    

\item Localization theorems of Borel-Atiyah-Segal type 
and GKM theory (with rational coefficients) for possibly singular 
complete $T$-varieties in characteristic zero. See \cite{go:opk} or the Appendix for details.  
\end{enumerate}

\smallskip

In \cite{f:sph}, Fulton, MacPherson, Sottile and Sturmfels succeed in
describing the non-equivariant operational Chow groups of 
complete spherical varieties. Indeed, they show that the 
Kronecker duality homomorphism 
$$
\mathcal{K}:A^i(X)\longrightarrow {\rm Hom}(A_i(X),\Z), \hspace{1cm} \alpha\mapsto (\beta\mapsto \deg(\beta\cap \alpha))
$$ 
is an isomorphism for complete spherical varieties. Here 
${\rm deg\,}(-)$ is the degree homomorphism $A_0(X)\to \Z$. 
Moreover, they prove that 
$A_*(X)$ is finitely generated by the classes of
$B$-orbit closures, and with the aid of the map 
$\mathcal{K}$, they provide a combinatorial description 
of $A^*(X)$ and the structure constants of the cap and cup products \cite{f:sph}. 
In addition, if $X$ is nonsingular and complete, 
they show that the cycle map $cl_X:A_*(X)\to H_*(X)$ is an isomorphism. 
Although we stated their results in the case of spherical varieties, these hold  
more generally for complete schemes with a finite number of orbits 
of a solvable group. In particular, the conclusions of \cite{f:sph} 
hold for Schubert varieties. 
Later on, Totaro \cite{to:linear} extended these results 
to 
the broader class of linear schemes, a class first 
studied in work of Jannsen \cite{jann:kth}. 
The results of \cite{f:sph} and \cite{to:linear} are quite marvelous in that 
they give a presentation of a rather abstract ring, namely $A^*(X)$, 
in a very combinatorial manner. 
%This is the  motivated the research in this article. 
%There is also an equivariant notion of $T$-linear scheme. 

\medskip

In this article, we extend the results of the previous paragraph      
to the equivariant Chow cohomology of  
$T$-linear schemes (Definition \ref{linear.dfn}).  
%a natural class of schemes with torus actions that generalizes linear schemes.     
%the natural 
%the equivariant analogues of linear schemes.    
%Briefly, a $T$-linear scheme is a $T$-scheme %with a $T$-action 
%that can be obtained by an 
%inductive procedure starting with a finite dimensional $T$-representation, 
%in such a way that the complement of a $T$-linear scheme equivariantly 
%embedded in affine space is also a $T$-linear scheme, and any $T$-scheme which
%can be stratified as a finite disjoint union of $T$-linear schemes  
%is a $T$-linear scheme. 
%See Section 2.3 for precise definitions.   
%$T$-linear schemes have been studied by Joshua-Krishna \cite{jk:chow}. 
By Theorem \ref{spherical_are_tlinear.thm}, 
%this   
%this class of $T$-schemes includes  
spherical varieties are $T$-linear (this fact does not follow directly 
from \cite{to:linear} and \cite{ro:sph},   
see the comments before Theorem \ref{spherical_are_tlinear.thm}).  
%In Section 3.1 we provide the equivariant versions of the results of \cite{f:sph} and \cite{to:linear}. 
%Building on this, 
%The equivariant versions of the results of \cite{f:sph} and \cite{to:linear} 
%are obtained in Subsection 3.1. 
%Building on our equivariant versions of the results of \cite{f:sph} 
%and \cite{to:linear}, 
%Our main results are: (i) a 
%Our main results are: (i) a 
Also, we obtain a 
localization theorem for the equivariant Chow cohomology 
of complete $T$-linear schemes that does not depend on resolution of 
singularities (Theorem \ref{eqloc.kro.thm}). 
Last, and most important, 
we give a presentation of the rational equivariant Chow cohomology 
of complete possibly singular 
spherical varieties admitting an equivariant smooth envelope 
(Theorem \ref{opA_spherical.thm}). %Remarkably,  
The latter vastly increases  the applicability 
of Brion's techniques     
\cite[Section 7]{bri:eqchow} 
from the smooth to the singular setup.

\smallskip

%This article is organized as follows. 
Here is an outline of the paper. 
Section 2 
reviews the necessary background material. %results from \cite{bri:eqchow}, \cite{f:sph} and \cite{to:linear}  
%needed in our study. 
Section 3 is the conceptual core of this article. 
In Subsection 3.1 we obtain the 
equivariant versions of 
the results of \cite{f:sph} and \cite{to:linear} that concern us. 
We start by defining {\em equivariant Kronecker duality schemes}. 
These are complete $T$-schemes $X$ which satisfy 
two conditions: (i) $A^T_*(X)$ is finitely generated over  
$S=A^*_T(pt)$, %the equivariant Chow ring of a point (denoted $S$),  
and (ii) the equivariant Kronecker duality map 
$\mathcal{K}_T:A^*_T(X)\longrightarrow \Hom_{S}(A_*^T(X),S)$    
%\hspace{1cm} \alpha\mapsto (\beta\mapsto p_{X*}{(\beta\cap \alpha)})$$
is an isomorphism of $S$-modules (Definition \ref{kro.dfn}). %See Section 3.1 for details. 
%Here $p_{X*}:A^T_*(X)\to S$ is the 
%map induced by pushforward to a point. Also, $S$ is isomorphic to $A^T_*(pt)$ with the opposite grading.   
%We refer to Section 3.1 for details. 
As an example, 
we show that complete 
$T$-linear schemes satisfy equivariant Kronecker duality (Proposition \ref{tlinear.kro.prop}). 
This  is 
%done by first   
%showing that $T$-linear schemes 
%satisfy 
deduced from the equivariant K\"{u}nneth formula (Proposition \ref{kunneth.prop}).  
%%a fact that follows almost immediately from the non-equivariant case. 
%Furthermore, we 
%describe the cap and cup product structures 
%(Corollaries \ref{action.op.chow.on.homology.cor} and \ref{op.cup.prod}). 
%These formulas are equivariant analogues of the ones in 
%\cite[Theorem 4]{f:sph}. 
In Subsection 3.2 we prove our second main result, namely, a localization theorem 
for equivariant Kronecker duality schemes (Theorem \ref{eqloc.kro.thm}). 
%Notably, it holds with $\Z$-coefficients. 
%This contrasts with the fact that, for more general possibly singular $T$-schemes, 
%such result only holds with $\Q$-coefficients (cf. Appendix).    
We conclude Section 3 by showing that projective group embeddings in 
arbitrary characteristic satisfy equivariant localization (Theorem \ref{Tlinear.envelopes.thm}). 
This extends well-known results on torus embeddings \cite{p:t} 
to more general compactifications of connected reductive groups.  
Finally, in Section 4, 
we apply the machinery just developed 
to spherical varieties in characteristic zero  
and prove the most important result of this paper,  
namely,    
Theorem \ref{opA_spherical.thm}. It asserts that 
if $X$ is a complete, possibly singular, 
$G$-spherical variety, then 
%our findings describe  
the image of the injective map 
$i^*_T:A^*_T(X)_\Q\to A^*_T(X^T)_\Q$ 
is fully described by congruences involving pairs, 
triples or quadruples of $T$-fixed 
points. %(Theorem \ref{opA_spherical.thm}). 
Remarkably, this extends \cite[Theorem 7.3]{bri:eqchow} 
to the singular setting. 

\smallskip

\acknowledge{ 
Most of the research in this paper was done 
during my first visit to the Institute
des Hautes \'Etudes Scientifiques (IHES). 
I am deeply grateful to IHES for its support,  
outstanding hospitality and excellent atmosphere. 
A very special thank you goes to Michel Brion 
for 
%the productive 
%meeting we had at IHES, from which this paper received much inspiration.  
carefully reading previous drafts of this article, 
and for offering detailed comments and suggestions on various parts.   
I would also like to thank Sabanc\i\,\"{U}niversitesi, 
the Scientific and Technological Research Council of Turkey (T\"{U}B\.ITAK), 
and the German Research Foundation (DFG) for their support during the 
final stages of completing this article. 
I am further grateful to the referee for very helpful comments and 
suggestions that improved the presentation and scope of the article. 
}

%%\section{Preliminaries}
\section{Definitions and basic properties}

%\convnotation{
\subsection{Conventions and notation} 
Throughout this paper, we fix an algebraically closed 
field $\k$ (of arbitrary characteristic, unless stated otherwise). 
All schemes and algebraic groups are assumed to be defined 
over $\k$. By a scheme we mean a separated scheme 
of finite type over $\k$. 
A variety is a reduced and irreducible scheme. 
%Observe that varieties need not be irreducible. 
A subvariety is a closed subscheme which is a variety. 
A point on a scheme will always be a closed point. 
The additive and multiplicative groups over $\k$
are denoted by $\mathbb{G}_a$ and $\G$. 

\smallskip

We denote by $T$ an algebraic torus. 
We write $\Delta$ for the character group of $T$, and  
$S$ for the symmetric algebra over $\Z$ of the abelian group 
$\Delta$. We denote by $\mathcal{Q}$ the quotient field of $S$. 
%
%\smallskip
%
A scheme $X$ provided with an algebraic action of $T$ is called a 
{\em $T$-scheme}. For a $T$-scheme $X$,  
we denote by $X^T$ the fixed point subscheme and by $i_T:X^T\to X$
the natural inclusion. %as a closed subscheme.  
If $H$ is a closed subgroup of $T$, we similarly denote by 
$i_H:X^H\to X$ the inclusion of the fixed point subscheme. 
When comparing $X^T$ and $X^H$ we write $i_{T,H}:X^T\to X^H$ 
for the natural ($T$-equivariant) inclusion. 

\medskip

A $T$-scheme $X$ is called {\em locally 
linearizable} (and the $T$-action is called {\em locally linear}) 
if $X$ is covered by invariant 
%quasi-projective open subsets 
%(and hence by invariant 
affine open subsets.  
%in the case of torus actions) 
%upon which the action is linear.    
For instance, $T$-stable subschemes of normal 
$T$-schemes are locally linearizable \cite{su:eq}.
A $T$-scheme is called {\em $T$-quasiprojective} if it 
has an ample $T$-linearized invertible sheaf. 
This assumption is satisfied, e.g. for $T$-stable subschemes   
of normal quasiprojective $T$-schemes \cite{su:eq}. 
%Similar notions are given for actions of more general (connected) linear  
%algebraic groups, see \cite{su:eq} for many details. 
Recall that an {\em envelope}  
$p:\tilde{X}\to X$ is a proper map such that for any 
subvariety $W\subset X$ there is a subvariety $\tilde{W}$
mapping birationally to $W$ via $p$ \cite[Definition 18.3]{f:int}.
In the case of $T$-actions, %(or more generally $G$-actions), 
we say that $p:\tilde{X}\to X$ is an {\em equivariant envelope} 
if $p$ is $T$-equivariant, and if we can take $\tilde{W}$ to be 
$T$-invariant for $T$-invariant $W$. If there is a dense open set  
$U\subset X$ over which $p$ is an isomorphism, then we say that 
$p:\tilde{X}\to X$ is a {\em birational} envelope. 
%These properties are compatible with Totaro's algebraic approximation
%to the Borel construction 
%For details see 
%\cite[Section 2.6]{eg:eqint}. 
%Similar notions exist for actions 
%of more general reductive groups $G$.  
By \cite[Theorem 2]{su:eq},    
%\begin{lem}\label{smooth_envelope.lem}
%Let $G$ be a connected linear algebraic group. 
if $X$ is a $T$-scheme, then there exists a 
$T$-equivariant birational envelope $p:\tilde{X}\to X$, 
where $\tilde{X}$ is a $T$-quasiprojective scheme. 
Moreover, if ${\rm char}(\k)=0$, then we may choose $\tilde{X}$ 
to be smooth \cite[Proposition 7.5]{eg:cycles}. 
%This holds more generally for actions 
%of linear algebraic groups, 
%see \cite[Theorem 2]{su:eq} and \cite[Proposition 7.5]{eg:cycles}.
If $p:\tilde{X}\to X$ is a $T$-equivariant envelope, and $H\subset T$ is a closed subgroup, 
then  the induced map 
$\tilde{X}^H\to X^H$ 
is a $T$-equivariant envelope \cite[Lemma 7.2]{eg:cycles}.

\medskip 

Let $X$ be a $T$-scheme of dimension $n$   
(not necessarily equidimensional). %, and let $d=\dim T$.    
Let $V$ be a finite dimensional $T$-module, %of dimension $l$, 
and let $U\subset V$ be an invariant open 
subset such that a principal bundle quotient 
$U\to U/T$ exists. %, and $V\setminus U$ has 
%codimension larger than $n-i$.  
Then $T$ acts freely on $X\times U$ 
%via the 
%diagonal action 
and the quotient scheme 
$X_T:=(X\times U)/T$ exists. 
Following Edidin and Graham \cite{eg:eqint}, 
we define the 
$i$-th equivariant Chow group  
$A_i^T(X)$ by $A_i^T(X):=A_{i+\dim U-\dim T}(X),$
if $V\setminus U$ has codimension more than 
$n-i$. Such 
pairs $(V,U)$ always exist, and  
the definition is independent of the choice of $(V,U)$, see       
\cite{eg:eqint}. Finally, $A^T_*(X):=\oplus_i A^T_i(X)$. 
Unlike ordinary Chow groups, 
$A^G_i(X)$ can be non-zero for any $i\leq n$, including negative $i$. 
If $X$ is a $T$-scheme, 
and $Y\subset X$ is a $T$-stable closed subscheme, 
then %the structure sheaf of 
$Y$ defines a class $[Y]$ 
in $A^T_*(X)$. If $X$ is smooth and equidimensional, then 
so is $X_T$, 
and $A^T_*(X)$ admits an intersection pairing;  
in this case, %we denote by $A^*_T(X)$  
the corresponding ring 
graded by codimension is isomorphic to 
the equivariant operational Chow ring $A^*_T(X)$ %, by Poincar\'e duality  
\cite[Proposition 4]{eg:eqint}. 
%the equivariant operational Chow group of $X$ 
The %(rational) 
equivariant Chow ring of a point $A^*_T(pt)$ identifies to $S$, 
and   
$A^T_*(X)$ is a $S$-module, where $\Delta$ acts on $A^T_*(X)$ 
by homogeneous maps of degree $-1$.   
This module structure 
is induced by pullback through the 
flat map $p_{X,T}:X_T\to U/T$.   
Restriction to a fiber of $p_{X,T}$ 
gives a canonical map 
$A^T_*(X)\to A_*(X)$, and this map is surjective (Theorem \ref{Tequiv.thm}). 
If $X$ is complete, we denote by 
$p_{X,T*}(\alpha)$ (or simply $p_{X*}(\alpha)$) 
the proper pushforward to a point of a class $\alpha \in A^T_*(X)$. 
%To simplify notation, 
We may also write $\int_X(\alpha)$ 
or $\deg{(\alpha)}$ for this pushforward.
%
%
%$\int_X(\alpha)$ 
%the proper pushforward to a point of a 
%class $\alpha \in A^T_*(X)$. 
%We usually write $A^T_*$ for the {\em negatively graded} $S$-module $A^T_*(pt)$.
Note that $A^*_T(pt)$  
is isomorphic to $A^T_*(pt)$ with the opposite grading.

\medskip

Let $X$ be a $T$-scheme. For any mixed space $X_T$ 
we construct a map $r:A^i_T(X)\to A^i(X_T)$.  
Let $c\in A^*_T(X)$. For a map $Y\to X_T$ and $\alpha \in A_*(Y)$, 
we define $r(c)\cap \alpha$ as follows. Let $Y_U\to Y$ be the pullback of the 
principal $T$-bundle $X\times U\to X_T$. Since $Y_U\to Y$ is a principal bundle, 
we identify $A_*(Y)$ with $A^T_*(Y_U)$. Let $\alpha_U\in A^T_*(Y_U)$ correspond to 
$\alpha\in A_*(Y)$. Now simply define $r(c)\cap \alpha$ to be the class 
corresponding to 
$c\cap \alpha_U$. See \cite[pages 620-621]{eg:eqint} for more information on the  
functorial properties of the map $r$.    
On the other hand, we also have a map $\rho:A^i(X_T)\to A^i_T(X)$. 
Indeed, let $c\in A^i(X_T)$, $Y\to X$ a $T$-equivariant map, and 
$\beta \in A^T_*(Y)$. For any representation, there are maps 
$Y_T\to X_T$. The class $\beta$ is represented by a class $\beta_U\in A_{*+\dim U-\dim T}(Y_T)$ 
for some mixed space $Y\times U/T$. Define 
$\rho(c)\cap \beta=c\cap \beta_U$. This  
is an element of $ A_{*+\dim U-\dim T-i}(Y_T)\simeq A^T_{*-i}(Y)$.  
Note that if $X$ has a $T$-equivariant smooth envelope 
(e.g. $X$ is a group embedding or ${\rm char}(\k)=0$),   
and $V\setminus U$ 
has codimension more than $i$, 
then $\rho$ and $r$ are inverse functions; so in this case we get  
$A^i_T(X)\simeq A^i(X_T)$ \cite[Theorem 2]{eg:eqint}. 

\medskip 

Finally, for any $T$-scheme $X$,  
%there is a natural map 
%$\iota^*:{\rm op}A^*_T(X)\to {\rm op}A^*(X)$. 
%Indeed, 
restriction to a fiber of 
$p_{X,T}:X_T\to U/T$   
induces a canonical map $A^*(X_T)\to A^*(X)$. 
%One defines $\iota=i^*\circ r$. 
%which coincides with $\iota^*$. 
Precomposing this map with $r:A^*_T(X)\to A^*(X_T)$ 
gives a natural map $\iota^*:A^*_T(X)\to A^*(X)$. 
In general, unlike its counterpart in equivariant Chow groups,  
the map $\iota^*$ is not surjective and %in general, and 
its kernel is not necessarily 
generated in degree one, not even 
for toric varieties \cite{pk:tor}. 
This becomes an issue %problems 
when translating results from 
equivariant to non-equivariant 
Chow cohomology. 
In Corollary \ref{usual.op.chow.cor} we give some  
conditions under which $\iota^*$ is surjective and 
yields an isomorphism 
$A^*_T(X)_\Q/\Delta A^*_T(X)_\Q\simeq A^*(X)_\Q$. %(see Corollary \ref{usual.op.chow.cor}). 
Such conditions are fulfilled, among others, by complete $\Q$-filtrable spherical varieties \cite{go:rm}.

\subsection{The Bia\l ynicki-Birula decomposition}
Let $X$ be a $T$-scheme. 
Let $X^T=\bigsqcup_{i=1}^m F_i$ be the decomposition of $X^T$ 
into connected components. 
A one-parameter subgroup $\lambda:\G\to T$ is called {\em generic}
if $X^{\G}=X^T$, where $\G$ acts on $X$ via $\lambda$. 
Generic one-parameter subgroups always exist 
(when $X$ is locally linearizable this certainly holds; 
the general case follows from this by considering the normalization of $X$).  
Now fix a generic one-parameter subgroup $\lambda$ of $T$. 
For each $F_i$, we define the subset %{\bf stratum} 
$$X_+(F_i,\lambda):=\{x \in X \;|\; \lim_{t\to 0}\lambda(t)\cdot x \;{\rm exists \;and\; is \;in}\; F_i\}.$$
We denote by $\pi_i:X_+(F_i,\lambda)\to F_i$ the map $x\mapsto \lim_{t\to 0}\lambda(t)\cdot x$. 
Then $X_+(F_i,\lambda)$ is a locally closed $T$-invariant subscheme of $X$, and 
$\pi_i$ is a $T$-equivariant morphism. 
The (disjoint) union of the 
%locally closed $T$-invariant subchemes 
$X_+(F_i,\lambda)$ 
%, where $\lambda$ is a fixed generic one-parameter subgroup, 
might not cover all of $X$,  
but when it does %this union encapsulates all of $X$    
(e.g. when $X$ is complete), 
the decomposition $\{X_+(F_i,\lambda)\}_{i=1}^m$ %$W_i(\lambda)$ 
is called 
the %associated 
Bia\l ynicki-Birula decomposition, or 
{\em BB-decomposition}, of $X$ associated to $\lambda$. 
Each $X_+(F_i,\lambda)$ is referred to as a {\em stratum} of the decomposition. %$X$. %(associated to $\lambda$).   
If, moreover, all fixed points 
of the given $T$-action %of $T$ 
on $X$ are isolated (i.e. $X^T$ is finite), 
the corresponding $X_+(F_i,\lambda)$ are simply called %{\bf BB-cells} or,  
{\em cells} %for short.
of the decomposition.

\begin{dfn}\label{filtable.dfn} 
Let $X$ be a $T$-scheme endowed with a BB-decomposition 
$\{X_+(F_i,\lambda)\}$, for some generic one-parameter subgroup $\lambda$ of $T$. 
The decomposition %$\{W_i(\lambda)\}$ 
is said to be {\em filtrable} if 
%$X$ is called {\bf filtrable} if, for some generic one-parameter subgroup $\lambda$ of $T$,           
%the collection $\{W_i(\lambda)\}$ is a BB-decomposition of $X$, 
%and  
there exists a finite increasing sequence 
$\Sigma_0\subset \Sigma_1\subset \ldots \subset \Sigma_m$
of $T$-invariant closed subschemes of $X$ such that:

\smallskip

\noindent a) $\Sigma_0=\emptyset$, $\Sigma_m=X$,

\smallskip

\noindent b) $\Sigma_{j}\setminus \Sigma_{j-1}$ is a stratum of 
the decomposition $\{X_+(F_i,\lambda)\}$, 
for each $j=1,\ldots, m$. %, t
\end{dfn}  
In this context, it is common to say that $X$ is {\em filtrable}.  
If, moreover, $X^T$ is finite and the cells $X_+(F_i,\lambda)$ are isomorphic to affine spaces $\A^{n_i}$,  
then $X$ is called {\em $T$-cellular}. 
%By assumption, there is a closed 
%stratum $X_1=W_{i_0}(\lambda)$ and $X\setminus X_1$ is filtrable (though it might not be projective). 
%
%\smallskip
%  
The following result is due to Bia\l ynicki-Birula (\cite{bb:torus}, \cite{bb:decomp}). 
%We include the proof here for the reader's convenience.

\begin{thm}\label{bbdecomp.thm}
Let $X$ be a complete $T$-scheme, and let $\lambda$ be a generic one-parameter subgroup. %with isolated fixed points.  
If $X$ admits an ample $T$-linearized invertible sheaf, 
then the associated BB-decomposition $\{X_+(F_i,\lambda)\}$ 
is filtrable. 
Furthermore, 
if $X$ is smooth,  
then $X^T$ is also smooth, and for any component $F_i$ of $X^T$,  %corresponding 
the map $\pi_i:X_+(F_i,\lambda)\to F_i$ 
makes $X_+(F_i,\lambda)$ into a $T$-equivariant locally trivial 
bundle in affine spaces over $F_i$. 

\hfill $\square$
\end{thm}

Hence, smooth projective $T$-schemes 
with isolated fixed points 
are $T$-cellular. 
% %On the other hand,  
% %in the category of $T$-quasiprojective schemes, 
% %projective $T$-schemes are always filtrable.

\subsection{$T$-linear schemes}
We introduce here the main objects of our study and outline some of their relevant
features. 

\begin{dfn} \label{linear.dfn}
Let $T$ be an algebraic torus and let $X$ be a $T$-scheme.
\begin{enumerate}
\item We say that $X$ is {\em $T$-equivariantly $0$-linear} 
if it is either empty or isomorphic
to ${\rm Spec}\,({\rm Sym}(V^*))$, where $V$ is a finite-dimensional rational representation of $T$. 

\item For a positive integer $n$, we say that $X$ is {\em $T$-equivariantly $n$-linear}
if there exists a family of $T$-schemes $\{U,Y,Z\}$, 
such that $Z\subseteq Y$ is a $T$-invariant closed immersion   
with $U$ its complement, $Z$ and one of the schemes $U$ or 
$Y$ are $T$-equivariantly $(n-1)$-linear and $X$ is the other
member of the family $\{U,Y,Z\}$.

\item We say that $X$ is {\em $T$-equivariantly linear} 
(or simply, {\em $T$-linear}) if it is $T$-equivariantly $n$-linear 
for some $n\geq 0$. $T$-linear varieties are varieties that are $T$-linear schemes. 
\end{enumerate}
\end{dfn}

%From the inductive definition of $T$-linear schemes  
It follows from the inductive definition 
that if $X$ is $T$-equivariantly $n$-linear, 
then $X^H$ is $T$-equivariantly $n$-linear, for any subtorus $H\subset T$. 
%provided $X$ is $T$-equivariantly $n$-linear.
Moreover, if $T\to T'$ 
is a morphism of algebraic tori, then every $T'$-linear scheme 
is also $T$-linear. 
Observe that $T$-linear schemes 
are {\em linear schemes} in the sense of Jannsen \cite{jann:kth} 
and Totaro \cite{to:linear}. 
While Totaro's class of linear schemes    
%, and \cite{jo:linear}. 
%Totaro's definition of linear schemes 
%The class of linear schemes considered in \cite{to:linear}  
%produces a subclass of that of Jannsen's 
is slightly narrower than that of Jannsen, the difference is nevertheless 
immaterial for our purposes. In fact, one easily checks   
that the main result of Totaro used here, 
namely \cite[Proposition 1]{to:linear}, 
holds for the larger class. %full generality. 
The following result is recorded in \cite{jk:chow}. 

\begin{prop}
Let $T$ be an algebraic torus and let $T'$ be a quotient of $T$. 
Let $T$ act on $T'$ via the quotient map. 
Then the following hold: 
\begin{enumerate}[(i)]
 \item $T'$ is $T$-linear. 
 \item A $T$-cellular scheme is $T$-linear.  
 \item Every $T$-scheme with finitely many $T$-orbits is $T$-linear. 
In particular, a toric variety with dense torus $T$ is $T$-linear. \hfill $\square$
\end{enumerate}
\end{prop}

It is well-known that 
%Due to the Bruhat decomposition, 
flag varieties, partial flag varieties 
and Schubert varieties come with a paving 
by affine spaces (due to the Bruhat decomposition), 
so they are all $T$-cellular   
and hence $T$-linear.  

\smallskip

Let $B$ be a connected solvable linear algebraic group with maximal torus $T$. 
A result of Rosenlicht \cite[Theorem 5]{ro:sph} shows  
that a homogeneous space for $B$ is isomorphic {\em as a variety} 
to $\mathbb{G}_{a}^r\times \G^s$, 
for some $r,s$. 
%Thus, if $X$ has finitely many $B$-orbits 
%(e.g. $X$ is spherical), then it is linear, 
As observed by Totaro \cite{to:linear}, 
this implies 
that a $B$-scheme with finitely many $B$-orbits (e.g. a spherical variety) 
is linear.   
Nevertheless, 
%since 
%we cannot conclude from this  
%\cite{ro:sph} 
this does not readily imply 
%it does not follow from this 
that such a scheme is $T$-linear, 
as  
Rosenlicht does not show    
that the isomorphism above %of $B$-orbits 
may be chosen $T$-equivariant. 
Presumably his arguments %proof of \cite[Theorem 5]{ro:sph} 
can be adjusted to achieve this.   
In any case, 
%since we cannot directly conclude from 
%\cite{ro:sph} that spherical varieties are $T$-linear, 
we shall give a direct proof of this 
fact, to keep the exposition self contained. %spherical varieties are $T$-linear. 
% 
%Notably, by a result of 
%$T$-linear schemes include a very rich class 
%of geometric objects: spherical varieties. 
%Next we present Brion's proof of this important fact. 

\begin{thm}\label{spherical_are_tlinear.thm}
%Spherical varieties are $T$-linear. 
%Assume ${\rm char}(\k)=0$. 
Let $B$ be a connected solvable linear algebraic group with maximal torus $T$.
Let $X$ be a $B$-scheme. If $B$ acts on $X$ with finitely 
many orbits, then $X$ is $T$-linear. In particular, spherical varieties are $T$-linear. %in arbitrary characteristic. 
\end{thm}

\begin{proof} %Brion's proof is as follows. 
The following argument was shown to the author by M. Brion (personal communication). 
Since $X$ is a disjoint union of $B$-orbits
and these are $T$-stable, it suffices to show that every $B$-orbit
is $T$-linear. 
Write this orbit as $B/H$ where $H$ is a closed
subgroup of $B$. 
Let $U$ be the unipotent radical of $B$. Then, 
we have a natural map $f : B/H \to B/UH$ and the right-hand side is a
torus (for it is a homogeneous space under the torus $T = B/U$).
Moreover, $f$ is a $B$-equivariant fibration with fiber $UH/H = U/(U \cap H)$, 
which is an affine space (as it is homogeneous under $U$).

We will show that the fibration $f$ is $T$-equivariantly trivial 
by factoring it into $T$-equivariantly
trivial fibrations with fiber the affine line. 
For this, we argue by induction on the dimension of $UH/H$. 
If this dimension is zero, then there is nothing to prove. 
If it is positive, then
$U$ acts non-trivially on $B/H$; 
replacing $U$ and $B$ with suitable
quotients, we may assume that $U$ acts faithfully. 
Because $U$ is a unipotent group normalized by $T$, 
we may find a one-dimensional unipotent subgroup $V$ 
of the center of $U$, which is normalized by
$T$ \cite[Lemma 6.3.4]{sp:lag}. 
So $V$ is isomorphic to 
%the additive group 
$\mathbb{G}_a$, and $T$ acts linearly on $V$ 
with some weight $\alpha$. 
By construction, %On the other hand, 
$V$ acts freely on $B/H$ via left multiplication, and the
quotient map is the natural map $B/H\to B/VH$, which is a
principal $V$-bundle. 
Since the variety $B/VH$ is affine (e.g by
the induction assumption) and $V\simeq \mathbb{G}_a$, this bundle
is trivial. 
%Note that $T$ acts linearly on $V$ 
%with some weight $\alpha$. 
%
%
%Hence $B/H$ is an affine variety and we have an
%isomorphism of coordinate rings  $k[B/H] = k[B/VH][t]$
%where $t$ is an indeterminate. 
%We claim that $t$ may be chosen an eigenvector of $T$. 
%To prove the claim, note that the additive group $V$ acts
%on $B/H$ and on $k[B/H]$ by automorphisms. 
%
%Because the principal $V$-bundle $B/H \to B/VH$ is trivial, the
%isomorphism $k[B/H] = k[B/VH][t]$ can be chosen so that $V$ acts on
%the right-hand side via its action by translations on $t$, that is, 
%$v \cdot t = v + t$
%(indeed, this is equivalent to having a $V$-equivariant isomorphism
%$B/H \cong B/VH \times V$). 
%So
%$B/H = B/VH \times V$,
%equivariantly for the actions of $T$ and $V$; here $T$ acts
%linearly on $V$ with weight $\alpha$. 
The isomorphism
$B/H = B/VH \times V$, equivariant for the action of $V$,
yields a regular function $g$ on $B/H$ such that
$g(v x) = v + g(x)$ for all $x$ in $B/H$ and $v$ in 
$V$ (identified to $\mathbb{G}_a$).
Let $T$ act on the ring of regular functions on $B/H$ via its action on
$B/H$ by left multiplication. 
We claim that $g$ may be chosen an eigenvector of $T$. 
Indeed, write $g$ as a sum of weight vectors
$g_{\lambda}$. Then for any $t$ in $T$, we obtain
$(t g) (v x) = g(t v x) = \alpha(t) v + (t g)(x)$
which yields
$\sum_{\lambda} \lambda(t) g_{\lambda}(v x) =
\alpha(t) v + \sum_{\lambda} \lambda(t) g_{\lambda}(x)$.
By viewing both sides as functions of $t$ and using linear independence
of characters, one gets
$g_{\alpha}(v x) = v + g_{\alpha}(x)$, and 
$g_{\lambda}(v x) = g_{\lambda}(x)$ for all $\lambda\neq \alpha$.
So these $g_{\lambda}$ are invariant under $V$, i.e.   
they are regular functions
on $B/VH$, and one may subtract them from $g$ to get $g= g_{\alpha}$. 
Now that the claim has been verified,  
%we conclude that 
%Hence, 
it follows that the product map $B/H\to B/VH \times V$,
where the second map is $g_{\alpha}$, yields the desired $T$-equivariant isomorphism,   
%
%
%
%
%Now %that the claim has been verified,  
and we conclude by induction.  
\end{proof}

%%%%%%%%%%%%%%%%%%%%%%%%%%%%%%%%%%%%%%%%%%% EQUIVARIANT INTERSECTION THEORY   %%%%%%%%%%%%%%%%%%%%%%%%%%%%%%%%%%%%%%%%%%%%
\subsection{Description of equivariant Chow groups}

Next we state %here 
Brion's presentation  
of the equivariant Chow groups of schemes with a torus action   
in terms of  
invariant cycles \cite[Theorem 2.1]{bri:eqchow}. 
It also shows how to recover usual Chow groups from equivariant ones.

\begin{thm}\label{Tequiv.thm}
Let $X$ be a $T$-scheme. Then the $S$-module $A^T_*(X)$ is defined by generators $[Y]$,  where 
$Y$ is an invariant subvariety of $X$, 
and relations $[{\rm div}_Y(f)]-\chi[Y]$ where $f$ is a 
rational function on $Y$ which is an eigenvector of $T$ of weight $\chi$. Moreover, the 
map $A^T_*(X)\to A_*(X)$ vanishes on $\Delta A^T_*(X)$, and it induces 
an isomorphism $A^T_*(X)/\Delta A^T_*(X)\to A_*(X).$
\hfill $\square$
\end{thm}

Now let $\Gamma$ be a connected solvable 
linear algebraic group with maximal torus $T$. 
If $X$ is a $\Gamma$-scheme, 
then the generators of $A^T_*(X)$ 
in Theorem \ref{Tequiv.thm} can be 
taken to be $\Gamma$-invariant 
\cite[Proposition 6.1]{bri:eqchow}. 
In particular, if $X$ has finitely many $\Gamma$-orbits (e.g. $X$ is spherical), 
then the $S$-module $A^T_*(X)$ is 
{\em finitely generated} by 
the classes of the $\Gamma$-orbit closures. %(this holds e.g. when $X$ is a spherical variety). 
%But 
% %This is a special case of the following lemma. 
%A more general statement is valid 
%for all $T$-linear schemes.  
More generally, one has the following lemma. 
%\medskip 

\begin{lem} \label{finite.rem}
Let $X$ be a $T$-linear scheme. Then the $S$-module $A^T_*(X)$ is 
finitely generated. In particular, $A_*(X)$ is a finitely generated abelian group. 
%Let $X$ be a $T$-linear scheme. 
Moreover, if $X$ is complete, then 
rational equivalence and algebraic equivalence coincide on $X$.  
\end{lem}

\begin{proof}
The first two assertions are easy consequences of the  
inductive definition of $T$-linear 
schemes (see e.g. \cite{to:linear}). %and the fact that for $0$-linear schemes,
Regarding the last one, observe that if $X$ is complete, then 
the kernel of the natural morphism 
$A_*(X)\to B_*(X)$ is divisible \cite[Example 19.1.2]{f:int},  
and thus trivial, for $A_*(X)$ is finitely generated. %by Lemma \ref{finite.rem}.  
\end{proof}

\medskip

Recall that if $X$ is a smooth equidimensional $T$-scheme, 
then 
$A^*_T(X)$ is isomorphic to the equivariant Chow group of $X$ graded 
by codimension \cite[Proposition 4]{eg:eqint}.

\begin{thm} [\protect{\cite{bri:eqchow}, \cite{vv:hkth}}] \label{smth.comp.free.rem} 
Let $X$ be a smooth $T$-variety. If $X$ is complete,      
then the $S_\Q$-module $A^T_*(X)_\Q$ is free. 
Moreover, the restriction homomorphism $i^*_T:A^*_T(X)_\Q\to A^*_T(X^T)_\Q$ 
is injective, and its image is the intersection of all the images of the 
restriction homomorphisms $i^*_{T,H}:A^*_T(X^H)_\Q\to A^*_T(X^T)_\Q$, where 
$H$ runs over all subtori of codimension one in $T$.  \qed 
\end{thm}

A few comments are in order here.  %Let $X$ be a smooth $T$-variety. 
First, Brion showed that Theorem \ref{smth.comp.free.rem} 
%$X$ is projective 
holds in the special case that $X$ is projective 
%Under the additional assumption that $X$ is projective, 
%Theorem \ref{smth.comp.free.rem} was first proved by Brion  
\cite[Theorems 3.2 and 3.3]{bri:eqchow}. 
%In the general case, %when $X$ is complete, 
%The general case   %Theorem \ref{smth.comp.free.rem} 
%follows from the results of Vezzosi and Vistoli 
%\cite{vv:hkth}. 
%In this reference, the authors 
%extended 
Later on, Vezzosi and Vistoli \cite{vv:hkth} generalized 
Brion's results to the setting of equivariant higher $K$-theory 
and 
%replacing the assumption of projectivity with that of completeness, 
established the corresponding %equivariant higher $K$-theory 
analogue of 
Theorem \ref{smth.comp.free.rem}, 
which holds for $X$ complete \cite[Corollary 5.11]{vv:hkth}. 
From this, by making appropriate changes in the proofs of 
\cite[Proposition 5.13 and Theorem 5.4]{vv:hkth}, 
%%i.e. replacing \cite[Lemma 4.2]{vv:hkth} with \cite[Proposition 3.2 (i)]{bri:eqchow}, 
% \cite{bri:eqchow}
one obtains Theorem \ref{smth.comp.free.rem} 
in its full form.    
%arguing as in \cite[Section 5]{vv:hkth}. 
The details can be found 
in a preprint version of \cite{vv:hkth} (arXiv version 3). 
Alternatively, see \cite[Sections 9 and 10]{kr:chow},  
where the results of \cite[Section 5]{vv:hkth} 
have been generalized  to equivariant higher Chow groups. 

\medskip

In characteristic zero  
Theorem \ref{smth.comp.free.rem} extends 
to %the rational equivariant Chow cohomology 
all possibly singular complete varieties \cite[Section 7]{go:opk}. See the Appendix for 
a review of the main results in this regard.

\medskip

The next lemma will become relevant later,   
when integrality of the 
equivariant operational Chow rings is discussed 
(cf. Lemma \ref{smooth.kro.comparison.lem}).    
It is essentially due to Brion, del Ba\~no and Karpenko.  
%a summary of the results of other authors. 

\begin{lem}\label{freeness.fix.global.lem}
Let $X$ be a smooth projective $T$-variety. Then the following 
are equivalent. 
\begin{enumerate}[(i)]
\item $A_*(X^T)$ is $\Z$-free. 
\item $A^T_*(X)$ is $S$-free.
\item $A_*(X)$ is $\Z$-free.
\end{enumerate}  
If moreover $X$ is $T$-linear, then any (and hence all) of these conditions hold. 
\end{lem}

\begin{proof}
The implication (i)$\Rightarrow$(ii) follows from \cite[Corollary 3.2.1]{bri:eqchow}, 
as any smooth projective variety is filtrable (Theorem \ref{bbdecomp.thm}). 
That (ii) implies (iii) is a 
consequence of Theorem \ref{Tequiv.thm}. 
To show that (iii) implies (i) we use    
a result of del Ba\~no \cite[Theorem 2.4]{ba:mot} and Karpenko \cite[Section 6]{kar:cell}. 
Namely, let $\lambda$ be a generic one-parameter subgroup of $T$, and let 
$X^T=\bigsqcup_i F_i$ be the decomposition of $X^T$ into connected components. 
Then, for every non-negative integer $j\leq \dim{(X)}$, %$0\leq j\leq  \dim{(X)}$, 
there is a natural isomorphism 
$\xymatrix{\bigoplus_i A^{j-d_i}(F_i)\ar^{\;\; \;\simeq}[r]& A^j(X)},$ 
where $d_i$ is the codimension of $X_+(F_i,\lambda)$ in $X$ 
(all spaces involved are smooth, 
so there is an intersection product on the Chow groups graded by codimension). 
These isomorphisms yield the assertion (iii)$\Rightarrow$(i).  
%\smallskip

Finally, 
if $X$ is a smooth projective $T$-linear variety, 
then, in particular, it is a projective linear variety, 
and so it satisfies the K\"{u}nneth formula (see below). 
Now Theorem \ref{es.thm} (ii) %yields freeness of $A_*(X)$ over $\Z$, that is,  
implies that condition (iii) of the lemma holds for $X$. This concludes the argument.  
\end{proof}

\medskip

For any schemes $X$ and $Y$, one has a K\"{u}nneth map  %in $\mathbf{Sch}_k$
$$A_*(X)\otimes A_*(Y)\to A_*(X\times Y),$$  
taking 
$[V]\otimes [W]$ to $[V\times W]$,  where $V$ and $W$ are subvarieties of $X$ and $Y$. 
This is an isomorphism only for very special schemes, e.g. linear 
schemes \cite[Proposition 1]{to:linear}; but when it is, 
strong consequences can be derived from it, as we shall see below.  
Let us start with     
the following result due to Ellingsrud and Stromme \cite[Theorem 2.1]{es:num}.  

\begin{thm}\label{es.thm}
Let $X$ be a smooth complete variety.  
Assume that the rational equivalence class $\delta$ of the diagonal $\Delta(X)\subseteq X\times X$  
is in the image of the K\"{u}nneth map
$A_*(X)\otimes A_*(X)\to A_*(X\times X).$
Let $\delta=\sum u_i\otimes v_i$  
be a corresponding decomposition of $\delta$,  
where $u_i,v_i\in A_*(X)$. 
Then  
\begin{enumerate}[(i)]
\item The $v_i$ generate $A_*(X)$, i.e. any $z\in A_*(X)$ has the form $\sum (u_i\cdot z)v_i$. %(equivalently, ),
%\medskip

\item Numerical and rational equivalence coincide on $X$.  %(cf. \cite[Example 19.1.4]{f:int}).  
In particular, $A_*(X)$ is a free $\Z$-module. 
%\medskip

\item If $k=\C$, then %the $v_i$ also generate $H_*(X,\Z)$, and 
the cycle map $cl_X:A_*(X)\to H_*(X,\Z)$ is 
an isomorphism. In particular, the homology and cohomology groups of $X$ 
vanish in odd degrees. \qed
\end{enumerate}  %\hfill $\square$  
\end{thm}

%%%%%%%%%5

Now consider a smooth complex algebraic variety $X$ with an action
of a complex algebraic torus $T$. Together with a cycle map 
$cl_X: A^*(X)\to H^*(X,\Z)$ 
(which doubles degrees \cite[Corollary 19.2]{f:int}),  
there is also an equivariant cycle map $cl_X^T: A^*_T(X)\to H^*_T(X,\Z)$ 
where $H^*_T(X,\Z)$ denotes
equivariant cohomology with integral coefficients, 
see \cite[Section 2.8]{eg:eqint}. 
%
%If $X$ is also complete 
%and the class of the diagonal is in the image of the K\"{u}nneth map, 
%then $cl_X$ is 
%an isomorphism (Theorem \ref{es.thm} (iii)).  
Next is a version of Theorem \ref{es.thm} (iii) for $cl_X^T$.

\begin{prop}\label{equiv.cycle.smooth.complete}
Let $X$ be a smooth  
complete complex 
%algebraic variety with a $T$-action. 
$T$-variety. 
If %$X$ is complete and 
the class of the diagonal $\Delta(X)\subseteq X\times X$ is in the image of the K\"{u}nneth map $A_*(X)\otimes A_*(X)\to A_*(X\times X)$, 
then the equivariant cycle map 
$$cl_X^T: A^*_T(X)\to H^*_T(X,\Z)$$
is an isomorphism. In particular, this holds %$cl_X^T$ is an isomorphism 
if $X$ is $T$-linear. %  variety.  
\end{prop}

\begin{proof}
In view of Theorem \ref{es.thm} (iii), 
the 
given hypothesis on $X$ imply that   
$cl_X$ is an isomorphism, and that   
$X$ has no integral cohomology in odd degrees. Then  
the spectral sequence associated to the fibration 
$X\times_{T}ET\to BT$ collapses, where $ET\to BT$ is the 
universal $T$-bundle. So the $S$-module   
$H^*_T(X,\Z)$ is free, %module over $S:=H^*_T(pt)$ 
and the map 
$H^*_T(X,\Z)/\Delta H^*_T(X,\Z)\to H^*(X,\Z)$ 
is an isomorphism. 
These results, together with the graded Nakayama Lemma, yield 
surjectivity of the 
equivariant 
cycle map $cl^T_X:A^*_T(X)\to H^*_T(X,\Z)$. 
To show injectivity, we proceed as follows. First, choose a basis $z_1,..,z_n$ of $H^*(X,\Z)$. Now 
identify that basis with a basis of $A^*(X)$ (via $cl_X$) and lift it to a generating system of 
the $S$-module $A^*_T(X)$. Then this generating system is a basis, since its image
under the equivariant cycle map is a basis of $H^*_T(X,\Z)$.

For the last assertion of the proposition, simply recall that 
if $X$ is $T$-linear, then the  
K\"{u}nneth map is an isomorphism \cite[Proposition 1]{to:linear}. %; 
%in particular, the class of $\Delta(X)$ is in the image of the K\"{u}nneth map.  
\end{proof}

\subsection{Equivariant Localization} 
Let $T$ be an algebraic torus. 
The following is the localization theorem for equivariant Chow groups \cite[Corollary 2.3.2]{bri:eqchow}.  

\begin{thm}\label{brion.loc.thm}
Let $X$ be a $T$-scheme.  
%and let $i_T:X^T\to X$ be the inclusion of the fixed point subscheme. 
If $X$ is locally linearizable, 
then 
the $S$-linear map $i_{T*}:A^T_*(X^T)\to A^T_*(X)$ is an 
isomorphism after inverting all non-zero elements of $\Delta$.  \hfill $\square$
\end{thm}

For later use, 
%we record here a slightly more general version of the previous 
%localization theorem. 
we prove a slightly more general statement. 

\begin{prop}\label{genloc.prop}
Let $X$ be a $T$-scheme, and let $H\subset T$ be a closed subgroup. 
%, and 
%let  % closed subgroup of  
%$i_H:X^H\to X$ be the inclusion of the fixed point subscheme.
Then the $S$-linear map  
$i_{H*}:A^T_*(X^H)\to A^T_*(X)$ 
becomes an isomorphism after inverting finitely many characters of $T$
that restrict non-trivially to $H$.  
\end{prop}

Before proving this proposition, we need a technical lemma.  
We would like to thank M. Brion for suggesting  
a simplified proof of this fact.    

\begin{lem} \label{technical.lem}
Let $X$ be an affine $T$-scheme. 
Let $H$ be a closed subgroup of $T$. 
Then the ideal of the fixed point subscheme $X^H$ 
is generated by all regular functions on $X$ which are 
eigenvectors of $T$ with a weight that restricts non-trivially to $H$.   
\end{lem}

\begin{proof}%simplified proof
Recall that $X^H$ is the largest closed
subscheme of $X$ on which $H$ acts trivially. In other words,
the ideal $I$ of $X^H$ is the smallest $H$-stable ideal of
$k[X]$ such that $H$ acts trivially on the quotient $k[X]/I$.
So $I$ is $T$-stable and hence the direct sum of its
$T$-eigenspaces. Moreover, if $f \in k[X]$ is a $T$-eigenvector
of weight $\chi$ which restricts non-trivially to $H$, then
$f \in I$. 
Indeed,  
let $\overline{f}$ be the image of $f$ in 
$k[X]/I$. 
Notice that $\overline{f}$ is a $T$-eigenvector of the same weight $\chi$ as $f$. %(as a simple check shows), 
Since $H$ acts trivially on $k[X]/I$, 
we obtain the identity  
$\overline{f}=h\cdot \overline{f}=\chi(h)\overline{f}$, 
valid for all $h\in H$. %holds. 
Nevertheless,  
there exists $h_0\in H$ such that $\chi(h_0)\neq 1$, 
by our assumption on $\chi$. 
Substituting this information into the above identity yields %, for some $h_0\in H$, % restricts non-trivially to $H$, %on the quotient, 
$\overline{f}=0$, equivalently, $f\in I$.  %, as one can easily check). 
Thus, $I$ contains the ideal $J$ generated
by all such functions $f$. 
But $k[X]/J$ is a trivial $H$-module by construction, 
and hence $I = J$ by minimality.
\end{proof}

\begin{proof}[Proof of Proposition \ref{genloc.prop}]
In virtue of Lemma \ref{technical.lem}, 
the proof 
is an easy adaptation %, almost word for word, 
of Brion's proof of 
\cite[Corollary 2.3.2]{bri:eqchow}, 
so we provide only a sketch of the crucial points. 
First, assume that $X$ is locally linearizable, i.e. 
$X$ is a finite union 
of $T$-stable affine open subsets $X_i$. 
Lemma 
\ref{technical.lem} implies that the ideal of each 
fixed point subscheme $X^H_i$ is generated 
by all regular functions on $X_i$ which are 
eigenvectors of $T$ with a weight that restricts 
non-trivially to $H$. Choose a finite set 
of such generators $(f_{ij})$, with respective 
weights $\chi_{ij}$. 
From Theorem \ref{Tequiv.thm} we know that the $S$-module $A^T_*(X)$ is 
generated by the classes of $T$-invariant subvarieties of $X$. 
%Moreover, by assumption, 
%
Now let $Y\subset X$ be a $T$-invariant subvariety of positive dimension. 
If $Y$ is not fixed pointwise by $H$, then one of the $f_{ij}$ 
defines a non-zero rational function on $Y$.  
Then, in the Chow group, we have $\chi_{ij}[Y]=[{\rm div}_Y{f_{ij}}]$. 
So after inverting $\chi_{ij}$, we get $[Y]=\chi_{ij}^{-1}[{\rm div}_Y{f_{ij}}]$. 
Arguing by induction on the dimension of $Y$, we obtain 
that $i_*$ becomes surjective after inverting the $\chi_{ij}$'s. 
A similar argument, using these $\chi_{ij}$'s 
in the proof of \cite[Corollary 2.3.2]{bri:eqchow}, 
shows that 
%Similarly, one checks that 
$i^*$ is injective after localization. %cf. \cite[Corollary 2.3.2]{bri:eqchow}. 

\smallskip

Finally, if $X$ is not locally linearizable, choose 
an equivariant birational envelope $\pi:\tilde{X}\to X$, 
where $\tilde{X}$ is normal (and possibly not irreducible).  
%be the disjoint union of the normalizations of the 
%irreducible components of $X$. 
%Then $\pi$ is an equivariant birational envelope.  
Let $U\subset X$ be the open subset where $\pi$ is 
an isomorphism.  
Set $Z=X\setminus U$ and $E=\pi^{-1}(Z)$. 
Then, by \cite[Lemma 2]{f:sph} 
and \cite[Lemma 7.2]{eg:cycles}, there is a commutative diagram 
$$
\xymatrix{
A^T_*(E^H)\ar[r]\ar[d]^{i_{H*}}& A^T_*(Z^H)\oplus A^T_*(\tilde{X}^H) \ar[r]\ar[d]^{i_{H*}}& A^T_*(X^H) \ar[r]\ar[d]^{i_{H*}}& 0\\
A^T_*(E)\ar[r]& A^T_*(Z)\oplus A^T_*(\tilde{X}) \ar[r]& A^T_*(X) \ar[r]& 0.
}
$$
Observe that $E$ and $Z$ have strictly smaller dimension than $X$. 
Moreover, $E$ and $\tilde{X}$ are locally linearizable. Applying Noetherian 
induction and the previous part of the proof, we get that   
the first two left vertical maps become isomorphisms 
after localization; hence so does the third one.    
\end{proof}

%The interested reader is invited to consult \cite{eg:eqint} and \cite{bloch:higher} 
%for more information 
%on equivariant and ordinary higher Chow groups. 

%%%%%%%%%%%%%%%%%%%%%%%%%%%%%%%%%  The first main part   %%%%%%%%%%%%%%%%%%%%%%%%%%%%%%%%%%%%%%%%%%%%%%%%%%%%%%%%%%%%%%%%%%%%%%%%%%

\section{Equivariant  
Kronecker duality and Localization}

\subsection{Equivariant Kronecker duality schemes}

\begin{dfn}\label{kro.dfn}
Let $X$ be a complete $T$-scheme. We say that $X$ satisfies 
{\em $T$-equivariant Kronecker duality} %(or Kronecker duality for short) 
if the following conditions hold:
\begin{enumerate}[(i)]
 \item $A^T_*(X)$ is a finitely generated $S$-module.
 \item The equivariant Kronecker duality map 
$$\mathcal{K}_T:A^*_T(X)\longrightarrow \Hom_{S}(A_*^T(X),S) \hspace{1cm} \alpha\mapsto (\beta\mapsto p_{X*}{(\beta\cap \alpha)})$$
is an isomorphism of $S$-modules. 
 \end{enumerate}
\end{dfn}

%In the same vein, 
Likewise, 
we say that $X$ satisfies {\em rational} $T$-equivariant Kronecker duality  
if 
%$A^T_*(X)_\Q$ is a finitely generated $S_\Q$-module and 
%the rational equivariant Kronecker duality map  
%$$\mathcal{K}_T:A^*_T(X)_\Q\longrightarrow Hom_{S_\Q}(A_*^T(X)_\Q,A^T_*(pt)_\Q)$$
%is an isomorphism of $S_\Q$-modules.   
the $S_\Q$-modules $A^T_*(X)_\Q$ and $A^*_T(X)_\Q$ 
satisfy the conditions (i) and (ii) with rational coefficients.

\begin{rem}\label{krfunct.rem} 
The equivariant Kronecker duality map 
is functorial for morphisms between complete $T$-schemes. 
Indeed,  
let $f:\tilde{X}\to X$ be an equivariant (proper) morphism of complete $T$-schemes.   
For any $\xi \in A^*_T(X)$, we have   
$$
\int_{\tilde{X}}f^*(\xi)\cap z=\int_{X}f_*(f^*(\xi)\cap z)=\int_{X}(\xi\cap f_*(z)), 
$$
due to the projection formula \cite{f:int}. 
%We claim that this identity 
This identity  
implies the commutativity 
of the diagram
$$
\xymatrix{
A^*_T(X)         \ar[rr]^{f^*}\ar[d]_{\mathcal{K}_T} & & A^*_T(\tilde{X}) \ar[d]^{\mathcal{K}_T}\\
{\rm Hom}_S(A^T_*(X),S)    \ar[rr]^{(f_*)^t}     &     & {\rm Hom}_S(A^T_*(\tilde{X}),S),
}
$$
where $(f_*)^t$ is the transpose of $f_*:A^T_*(\tilde{X})\to A^T_*(X)$. 
%Indeed, at the level of elements, we have   
%$$
%\xymatrix{
%\alpha \ar@{|->}[r]\ar@{|->}[d] &p^*(\alpha) \ar@{|->}[d]\\
%\{\beta\mapsto \int_X \alpha\cap \beta\} \ar@{|->}[r]& \{\theta\mapsto \int_X \alpha\cap p_*(\theta)=\int_X p_*(p^*(\alpha)\cap \theta)=\int_{\tilde{X}}p^*(\alpha)\cap \theta\},\\
%}
%$$
%and the claim follows.  
\end{rem}

\smallskip

It follows from Definition \ref{kro.dfn} that if 
$X$ satisfies $T$-equivariant Kronecker duality, 
then 
the $S$-module
$A^*_T(X)$ is finitely generated and torsion free. 
In particular, if $T$ is one dimensional, i.e. $T=\G$, 
then 
%$\G$-equivariant Kronecker duality implies 
%that 
$A^*_{\G}(X)$ is a finitely generated 
free module over $A^*_{\G}=\Z[t]$. 
Moreover,    
%If moreover 
if 
$X$ is projective and smooth, then 
$A_*(X^{\G})$ is a finitely generated free abelian group 
(Lemma \ref{freeness.fix.global.lem}). 

\smallskip

As it stems from the 
previous paragraph, not all smooth varieties with a torus action satisfy 
Equivariant Kronecker duality. For a more concrete example, 
consider  
the trivial action of $T$ on a projective smooth curve. 
In this case, one checks that $\mathcal{K}_T$ is an extension of the 
non-equivariant Kronecker duality map $\mathcal{K}$. But, as pointed 
out in \cite{f:sph}, the kernel of $\mathcal{K}$ in degree one 
is the Jacobian of the curve, which is non-trivial if the curve 
has positive genus.

\begin{lem}\label{smooth.kro.comparison.lem}
Let $X$ be a smooth complete $T$-variety. Then $X$ satisfies 
rational $T$-equivariant Kronecker duality if and only if 
it satisfies the rational non-equivariant Kronecker duality, i.e. 
$\mathcal{K}:A^i(X)_\Q\to \Hom(A_i(X),\Q)$ is an isomorphism for all $i$. 
If, moreover, $X$ is projective and $A_*(X^T)$ is $\Z$-free, then the equivalence holds over 
the integers.   
\end{lem}

\begin{proof} 
Both assertions are proved similarly, 
%interchanging Lemma \ref{freeness.fix.global.lem} with  
%Remark \ref{smth.comp.free.rem} 
%at appropriate places,  
so we focus on the second one. 
Since $X$ is smooth and projective, the 
assumption on $A_*(X^T)$ implies 
that $A^T_*(X)$ is a free $S$-module (Lemma \ref{freeness.fix.global.lem}; cf. Theorem  \ref{smth.comp.free.rem}). 
Now, by Poincar\'e duality \cite[Proposition 4]{eg:eqint},  
$A^*_T(X)$ is isomorphic to $A^T_*(X)$;    
so it is also a free $S$-module. 
By the graded Nakayama lemma, 
$\mathcal{K}_T$ is an isomorphism 
if and only if 
$$
\overline{\mathcal{K}_T}: A^*_T(X)/\Delta A^*_T(X) \to {\rm Hom}_S(A^T_*(X),S)/\Delta {\rm Hom}_S(A^T_*(X),S)
$$
is an isomorphism. But
freeness of $A^T_*(X)$ 
yields an isomorphism
$$
{\rm Hom}_S(A^T_*(X),S)/\Delta {\rm Hom}_S(A^T_*(X),S)
\simeq {\rm Hom}(A^T_*(X)/\Delta A^T_*(X),\Z)
$$ 
and the later identifies to 
${\rm Hom}(A_*(X),\Z)$, by Theorem \ref{Tequiv.thm}. 
On the other hand, by Theorem \ref{Tequiv.thm} again, 
the map 
$
A^*_T(X)/\Delta A^*_T(X)\to A^*(X) 
$
is an isomorphism. 
These facts, together with the commutativity of the diagram  below 
$$\xymatrix{
A^*_T(X) \ar[rr]^{\mathcal{K}_T\;\;\;\;\;\;\;\;}\ar[d]   &  &{\rm Hom}_S(A^T_*(X),S)\ar[d] \\
%A^*(X)=A^*_T(X)/MA^*_T(X)   \ar[r]^{\mathcal{K}\;\;\;\;\;\;\;\;\;\;\;\;} &{\rm Hom}(A^T_*(X)/MA^T_*(X),\Z)={\rm Hom}(A_*(X),\Z). 
A^*(X)   \ar[rr]^{\mathcal{K}\;\;\;\;\;\;\;\;\;\;\;\;} & &{\rm Hom}(A_*(X),\Z), 
} $$
%state that $\overline{\mathcal{K}_T}$ (and so $\mathcal{K}_T$) is an isomorphism 
%if and only if $\mathcal{K}$ is an isomorphism. This last equivalence is exactly 
yield the content of the lemma. 
\end{proof}

%Our goal is to 
Next we show that complete $T$-linear 
schemes satisfy equivariant Kronecker duality. 
For this, the main ingredient is the following result, due to  
Totaro \cite{to:linear} %and Joshua \cite{jo:linear} 
in the non-equivariant case. 
%, and to Joshua and Krishna \cite{jk:chow} in the case of equivariant $K$-theory. 
%It states that $T$-linear schemes 
%satisfy the equivariant K\"{u}nneth 
%formula. 
Recall that a $T$-scheme $X$ 
is said to satisfy the 
{\em equivariant K\"{u}nneth formula} if 
the K\"{u}nneth map (or exterior product \cite{eg:eqint}) 
$$
A^T_*(X)\otimes_SA^T_*(Y)\to A^T_*(X\times Y)
$$
is an isomorphism for {\em any} $T$-scheme $Y$. %is any $T$-scheme.  
%Although a proof of this fact  
%can be obtained using equivariant motivic cohomology (and 
%will appear in \cite{jk:motivic}),  
%we give here a somewhat simpler proof. 

\begin{prop}\label{kunneth.prop}%\cite{jk:chow}
Let $X$ be a $T$-scheme. 
If $X$ is $T$-linear, then 
it satisfies the equivariant K\"{u}nneth formula. 
\end{prop}

\begin{proof} 
When $X$ is $T$-linear, one can choose representations $V$ of $T$ 
%(refer to Section 2.1 for notation) 
so that $X_T$ is linear, see e.g. \cite[Section 2.2]{bri:eqchow} 
and \cite[Section 1]{p:t}. Now the result follows from \cite[Proposition 1]{to:linear}.
\end{proof}

\begin{prop}\label{tlinear.kro.prop}
%Let $X$ be a complete $T$-scheme. 
If $X$ is a complete $T$-linear scheme, then 
%If $X$ is a projective $T$-linear scheme, then 
the equivariant Kronecker map 
$$
\mathcal{K}_T:A^*_T(X)\to {\rm Hom}_S(A^T_*(X),S)
$$
is an isomorphism. 
\end{prop}

This result follows quite formally from Proposition \ref{kunneth.prop}, 
as in the non-equivariant case \cite[Theorem 3]{f:sph},    
so we only sketch the proof. %outline it. %of the main point. 
%\begin{proof}[Sketch of proof] 
To define the inverse to $\mathcal{K}_T$,   
given a $S$-module homomorphism $\varphi:A^T_*(X)\to S$,   
we construct an element 
$c_{\varphi}\in A^*_T(X)$. %as follows. 
Since the $S$-module $A^T_*(X)$ is finitely generated, 
we can assume, without loss of generality, 
that $\varphi$ is homogeneous  
%Indeed, any 
%$\varphi\in {\rm Hom}_S(A^T_*(X),S)$ decomposes as $\varphi=\sum \varphi_\tau$, 
%where $\varphi_\tau$ is a homogeneous homomorphism of degree $\tau$  
\cite[Part II, Section 11.6]{bourbaki:mod}.   
Bearing this in mind, 
given a homomorphism $\varphi:A^T_*(X)\to S$ of degree $-\lambda$, 
we build $c_\varphi\in A^\lambda_T(X)$ as follows.  
For a $T$-map $f:Y\to X$, 
the corresponding homomorphism $f^*c_{\varphi}:=
c_{\varphi}(f):A_*^T(Y)\to A_{*-\lambda}^T(Y)$ 
is defined to be the composite
%\begin{eqnarray*}
%\begin{footnotesize}
$$
\xymatrix@C=1.6em@R=2em{%@C=0.2em@W=0.2em@R=0.5em{
A^T_*(Y)\ar[r]^{(\gamma_f)_*\hspace{.5cm}} & 
A^T_*(X\times Y)\ar[r]^{\simeq\hspace{.5cm}}& A^T_*(X)\otimes_S A^T_*(Y) \ar[r]^{\varphi\otimes {\rm id}\;\;\;\;}& %\displaystyle \sum_{j\in \Z}A^T_{j}(X)\otimes A^T_{m-j}(Y) \ar[r]^{{\rm pr}_{\leq -\lambda}} & 
%%% A^T_{-\lambda}(X)\otimes A^T_{*+\lambda}(Y) \ar[d]^{\varphi\otimes {\rm id}} \\
%%%%%%\displaystyle \sum_{i\leq -\lambda}A^T_{i}(X)\otimes A^T_{m+i}(Y) \ar[d]^{\varphi\otimes {\rm id}} \\
%%% & &  & \hspace{1cm} \Q\otimes A^T_{*+\lambda}(Y), \\}%\simeq A^T_{*-i}(Y).\\
% & &  & \hspace{1cm} A^T_{m+\lambda}(Y). \\}
                   S\otimes_S A^T_{*}(Y)\simeq A^T_{*}(Y)_,}\\ %should *=(*-\lambda)?
$$ 
%\end{footnotesize}
where $(\gamma_f)_*$ denotes the proper pushforward along the graph of $f$, 
and the second displayed map is the K\"{u}nneth isomorphism (Proposition \ref{kunneth.prop}). 
The verification that $c_{\varphi}$ %(for different $Y$'s) 
satisfies the compatibility axioms, %Chapter 17, 
and that this construction indeed 
gives the inverse to $\mathcal{K}_T$,  
is the same as in \cite{f:sph}. 
%\end{proof}

\medskip

For $c\in A^\lambda_T(X)$ and $z\in A^T_*(X)$, we write $c(z)$ for ${\rm deg}(c\cap z)$. 
The next two corollaries 
describe the cap and cup product structures. They  
are easily deduced from the previous proposition, cf.  
\cite[Corollaries 1 and 2]{f:sph}. %; therefore we skip the proofs.  
%was obtained along the proof of the previous proposition. 

\begin{cor}\label{action.op.chow.on.homology.cor}
Let $f:Y\to X$, $c\in A^\lambda_T(X)$, $z\in A^T_m(Y)$. 
Suppose $(\gamma_f)_*(z)=\sum u_i\otimes v_i$ with $u_i\in A^T_{p(i)}(X)$
and $v_i\in A^T_{m-p(i)}(Y)$. 
Then $f^*c\cap z=\sum_{p(i)\leq \lambda}c(u_i)v_i.$ \qed %\xqedhere{128.5pt}{\qed} %\hfill $\square$
\end{cor}

\begin{cor}\label{op.cup.prod}
Let $c\in A^\lambda_T(X)$, $c'\in A^\mu_T(X)$, and  $z\in A^T_m(X)$, where $m\leq \lambda+\mu$.  
Write $\delta_*(z)=\sum u_i\otimes v_i$ 
with $u_i\in A^T_{p(i)}(X)$ and $v_i\in A^T_{m-p(i)}(X)$. Then 
$
(c\cup c')(z)=\sum_{m-\mu \leq p(i)\leq \lambda} c(u_i)c'(v_i).
$  \qed 
\end{cor}

\medskip 

For a $T$-scheme $X$, 
there is a natural map $\iota^*:A^*_T(X)\to A^*(X)$ (Section 2.1). 
In general, when $X$ is singular, $\iota^*$ may not be surjective, and 
its kernel may not be generated in degree one \cite{pk:tor}. 
%But when $X$ is singular, this is hardly the case. 
Next we describe a class of possibly singular $T$-schemes 
for which the map $\iota^*$ is well-behaved.    
%for possibly singular $X$ under certain conditions. 
This yields   
the compatibility of our product formulas 
with those of \cite{f:sph}. %, whenever $A^*_T(X)$ is $S$-free. 

\begin{cor}\label{usual.op.chow.cor}
%Assume ${\rm char}(\k)=0$. 
Let $X$ be a complete $T$-scheme. 
If $X$ is $T$-linear and $A^T_*(X)$ is $S$-free, 
then %$i^*:A^*_T(X)\to A^*(X)$ is surjective, 
%and 
the map 
$A^*_T(X)/\Delta A^*_T(X)\to A^*(X)$, induced by $\iota^*$, 
is an isomorphism.  
%where $\Delta$ is the character group of $T$.   
\end{cor}

\begin{proof}
%Use Proposition 5 of E-G Localization in the Chow group, and the commutative square of Lemma 3 in 
%Brion's Poincar\'e duality paper.
Proposition \ref{tlinear.kro.prop} together with freeness of $A^T_*(X)$ yield  
%%\begin{small}
$$
\xymatrix@C=1.2em@R=0.5em{%%@C=0.2em@W=0.2em@R=0.5em{
A^*_T(X)/\Delta A^*_T(X) & \simeq & {\rm Hom}_S(A^T_*(X),S)/\Delta {\rm Hom}_S(A^T_*(X),S) \\
                   & \simeq & {\rm Hom}_{\Z}(A^T_*(X)/\Delta A^T_*(X),\Z). 
}
$$
Furthermore, by Theorem \ref{Tequiv.thm}, the term on the right hand side 
corresponds to 
${\rm Hom}(A_*(X),\Z)$, which, in turn, is isomorphic to $A^*(X)$, due to the 
non-equivariant version of Kronecker duality \cite[Proposition 1]{to:linear}. %, 
Considering this information  alongside the commutative diagram   
$$\xymatrix{
A^*_T(X) \ar[rr]^{\mathcal{K}_T\;\;\;\;\;\;\;\;}\ar[d]   &  &{\rm Hom}_S(A^T_*(X),S)\ar[d] \\
%A^*(X)=A^*_T(X)/MA^*_T(X)   \ar[r]^{\mathcal{K}\;\;\;\;\;\;\;\;\;\;\;\;} &{\rm Hom}(A^T_*(X)/MA^T_*(X),\Z)={\rm Hom}(A_*(X),\Z). 
A^*(X)   \ar[rr]^{\mathcal{K}\;\;\;\;\;\;\;\;\;\;\;\;} & &{\rm Hom}(A_*(X),\Z), 
} $$
produces the content of the corollary. 
\end{proof}

The conditions of 
Corollary \ref{usual.op.chow.cor} 
are satisfied by possibly  
singular $T$-cellular varieties (e.g. Schubert varieties). 
With $\Q$-coefficients, the corresponding statement   
is satisfied by  
$\Q$-filtrable 
spherical varieties \cite{go:rm}. 
This class includes all rationally smooth projective 
equivariant embeddings of reductive groups \cite{go:rm}. %({\em op. cit.}). 
%Compare 
%\cite{go:cells}.  

\subsection{Localization for $T$-equivariant Kronecker duality schemes}

%We explore some of the appealing features of 
%$T$-equivariant Kronecker duality schemes, 
%a class of spaces that, in view of our previous results, 
%includes all complete $T$-linear schemes.

From the viewpoint of algebraic torus actions, %on schemes,
the main attribute of equivariant Kronecker duality schemes 
is that they supply a somewhat more intrinsic  
background for 
establishing 
localization theorems on 
integral equivariant Chow cohomology.

\begin{thm}\label{eqloc.kro.thm}
Let $X$ be a complete $T$-scheme satisfying $T$-equivariant Kronecker duality. 
Let $H\subset T$ be a subtorus of $T$ and let 
$i_{H}:X^{H}\to X$ be the inclusion of the fixed point 
subscheme. 
If $X^{H}$ also satisfies $T$-equivariant Kronecker duality, 
then the morphism   
$$
i^*_{H}:A^*_T(X)\to A^*_T(X^{H})
$$ 
becomes an isomorphism after inverting 
finitely many characters of $T$ that restrict non-trivially 
to $H$. In particular, $i^*_{H}$ is injective over $\Z$. 
\end{thm}

\begin{proof}
By Proposition \ref{genloc.prop},  
the localized map  
$(i_{H}*)_{\mathcal{F}}:A^T_*(X^{H})_{\mathcal{F}}\to 
A^T_*(X)_{\mathcal{F}}$ 
is an isomorphism, where $\mathcal{F}$ is a finite family 
of characters of $T$ that restrict non-trivially to $H$.

Now consider the commutative diagram 
$$
\xymatrix{
A^*_T(X)            \ar[rr]^{i^*_{H}}\ar[d]      & & A^*_T(X^{H})  \ar[d]       \\
\Hom_S{(A^T_*(X),S)}  \ar[rr]^{(i_{{H}*})^t} & & \Hom_S{(A^T_*(X^{H}),S)}, \\
}
$$ 
where $(i_{{H}*})^t$ represents the transpose %dual 
of $i_{{H}*}:A^T_*(X^{H})\to A^T_*(X)$ (commutativity 
follows from %the projection formula 
Remark \ref{krfunct.rem}, because $i_{H}$ is proper).   
By our assumptions on $X$ and $X^{H}$, 
both vertical maps are isomorphisms. 
Moreover, 
after localization at $\mathcal{F}$, 
the above commutative diagram becomes
$$ 
\xymatrix{
A^*_T(X)_{\mathcal{F}}                    \ar[rr]^{(i^*_{H})_{\mathcal{F}}}\ar[d] &  & A^*_T(X^{H})_{\mathcal{F}}  \ar[d]       \\
(\Hom_S{(A^T_*(X),S)})_{\mathcal{F}}  \ar[rr]^{((i_{{H}*})^t)_{\mathcal{F}}}  & & (\Hom_S(A^T_*(X^{H}),S))_{\mathcal{F}}. \\
}
$$

Since %$A^*_T(X)\simeq \Hom{(A^T_*(X),A^*_T)}$ and 
$A^T_*(X)$ is a finitely generated $S$-module (as $X$ satisfies equivariant Kronecker duality), 
localization commutes with formation of $\Hom$  
(see \cite[Prop. 2.10, p. 69]{ei:comm}), and so 
$$A^*_T(X)_{\mathcal{F}}\simeq 
(\Hom_S{(A^T_*(X),S)})_{\mathcal{F}}\simeq 
\Hom_{S_\mathcal{F}}(A^T_*(X)_{\mathcal{F}}, {S}_\mathcal{F}).$$
Similarly, for $X^{H}$ we obtain
$$
A^*_T(X^{H})_{\mathcal{F}}\simeq (\Hom_S(A^T_*(X^{H}),S))_{\mathcal{F}} \simeq \Hom_{S_\mathcal{F}}(A^T_*(X^{H})_{\mathcal{F}}, {S}_\mathcal{F}). 
$$
In other words, the bottom map from the preceding diagram fits in the commutative square 
$$
\xymatrix{
(\Hom_S{(A^T_*(X),S)})_{\mathcal{F}}  \ar[rr]^{((i_{{H}*})^t)_{\mathcal{F}}} \ar[d] & & (\Hom_S(A^T_*(X^{H}),S))_{\mathcal{F}}\ar[d] \\
\Hom_{S_\mathcal{F}}(A^T_*(X)_{\mathcal{F}}, {S}_\mathcal{F})                   \ar[rr]^{((i_{{H}*})_{\mathcal{F}})^t} & &  \Hom_{S_\mathcal{F}}(A^T_*(X^{H})_{\mathcal{F}}, {S}_\mathcal{F}), \\
}
$$
where the vertical maps are natural isomorphisms. %(\cite[Prop. 2.10, p. 69]{ei:comm}).  
But we already know that $(i_{{H}*})_{\mathcal{F}}$ is an isomorphism, hence 
so are $((i_{{H}*})_{\mathcal{F}})^t$, 
$((i_{{H}*})^t)_{\mathcal{F}}$ and $(i^*_{H})_{\mathcal{F}}$. 

\smallskip

Finally, to prove the last assertion of the theorem, recall that 
the $S$-module 
$A^*_T(X)$ 
is finitely generated and torsion free (Definition 3.1). Hence the natural map 
$A^*_T(X)\to A^*_T(X)\otimes_S \mathcal{Q}$ is injective, 
where $\mathcal{Q}$ is the quotient field of $S$. 
In particular, the (also natural) map  
$A^*_T(X)\to A^*_T(X)_{\mathcal{F}}$ 
is injective. This, together with the 
first part of the theorem, yields injectivity of $i^*_{H}$. We are done. 
\end{proof}

\begin{cor}\label{cs.cor.tlinear}
Let $X$ be a complete $T$-scheme. 
Let $H$ be a codimension-one subtorus of $T$. 
%, and let $i_H:X^H\to X$ 
%be the natural inclusion. 
If $X$ is $T$-linear, then the pullback 
$i^*_H:A^*_T(X)\to A^*_T(X^H)$ is injective over $\Z$.  
\end{cor}
\begin{proof}
If $X$ is $T$-linear, 
then 
so is $X^H$. Now use 
Proposition \ref{tlinear.kro.prop} and Theorem \ref{eqloc.kro.thm}.  
\end{proof}

\smallskip

%\begin{rem}
Let $X$ be a complete $T$-linear scheme. It 
follows from Corollary \ref{cs.cor.tlinear} that 
the image of the {\em injective} 
map $i^*_T:A^*_T(X)\to A^*_T(X^T)$ 
is contained in the intersection 
of the images of the (also injective) maps
$
i^*_{T,H}:A^*_T(X^H)\to A^*_T(X^T),
$
where $H$ runs over all subtori of codimension one in $T$. 
%In symbols, 
%${\rm Im}(i^*_T)%:{\rm op}A^*_T(X)_\Q\to {\rm op}A^*_T(X^T)_\Q]
%\subseteqq 
%\bigcap_{H\subset T} {\rm Im}(
%i^*_{T,H}).$
%where the intersection runs over all codimension-one subtori $H$ of $T$.
%
When the image of $i^*_T$ is {\rm exactly} the intersection of 
the images of the maps $i^*_{T,H}$ we say, following  
\cite{go:opk}, that $X$ has the 
{\em Chang-Skjelbred property} (or {\em CS property}). 
If the defining condition holds over $\Q$ rather than $\Z$, we say that 
$X$ has the {\em rational CS property}.
By Theorem \ref{smth.comp.free.rem}, any smooth complete $T$-scheme has 
the rational CS property; by Theorem \ref{cs.thm},     
so does any complete $T$-scheme in characteristic zero. 
It would be interesting to determine, in arbitrary characteristic,   
which complete, possibly singular,   
$T$-linear schemes satisfy the CS property. %, in arbitrary characteristic. 
For instance, toric varieties are 
known to have this property \cite{p:t}.  
We anticipate that this also holds for 
projective embeddings of semisimple groups 
of adjoint type (this shall be pursued elsewhere). 
For the corresponding problem with rational coefficients, we provide an answer next.     
%As for the rational CS property, 
%we provide a result in this direction.   
%Notice that for equivariant operational $K$-theory, 
%the corresponding 
%CS property holds over $\Z$ for 
%every complete $T$-scheme in characteristic zero \cite{go:opk}.   
%\end{rem}

\begin{thm}\label{Tlinear.envelopes.thm} 
Let $X$ be a complete $T$-linear scheme. 
If there exists an equivariant envelope $\pi:\tilde{X}\to X$  
%from a smooth $T$-scheme $\tilde{X}$, 
with $\tilde{X}$ smooth, 
then $X$ has the rational CS property. 
In particular, projective embeddings 
 of connected reductive linear algebraic groups 
 have the rational CS property in arbitrary characteristic.
\end{thm}

\begin{proof}
Let $u\in A^*_T(X^T)_\Q$ be such that   
$u\in \bigcap_{H\subset T} {\rm Im}(i^*_{T,H})$,  
%:{\rm op}A^*_T(X^{H})_\Q\to {\rm op}A^*_T(X^T)_\Q]$, 
where the intersection runs over all codimension-one subtori $H$ of $T$. 
Our task is to show that $u\in {\rm Im}(i^*_T)_\Q$. 
First, observe that there is a commutative diagram   
$$
\xymatrix{
 & A^*_T(X)_\Q   \ar[rr]^{\pi^*}\ar[d]^{i_T^*}\ar[ddl]_{i^*_H}& 
&A^*_T(\tilde{X})_\Q \ar[d]^{\widetilde{i_T^*}}\ar[ddl]_{\widetilde{i^*_H}}\\
 & A^*_T(X^T)_\Q \ar[rr]^{{\pi_T}^*}& &A^*_T(\tilde{X}^T)_\Q\\
A^*_T(X^H)_\Q \ar[rr]^{{\pi_H}^*}\ar[ur]_{i^*_{T,H}}& & 
A^*_T(\tilde{X}^H)_\Q\ar[ur]_{\widetilde{i^*_{T,H}}} & \\
}
$$
obtained by combining 
and comparing  
the sequences that \cite[Lemma 7.2]{eg:cycles} and Theorem \ref{kimura.thm} assign   
to the envelopes $\pi:\tilde{X}\to X$, $\pi_H:\tilde{X}^H\to X^H$ 
and $p_T:\tilde{X}^T\to X^T$ (cf. \cite[proof of Theorem 4.4]{go:opk}. 
From the diagram it follows that 
%if $u\in {\rm op}A(X^T)_\Q$ is in 
%the image of $i^*_{T,H}$, then 
${\pi_T}^*(u)$ is in the image of $\widetilde{i^*_{T,H}}$. 
Hence, %if $u$ is in the intersection of the images of all $i^*_{T,H}$, then 
${\pi_T}^*(u)$ is in the intersection of the images of all 
$\widetilde{i^*_{T,H}}$, where $H$ runs over all codimension-one subtori of $T$. 
Since $\tilde{X}$ is known to have 
the rational CS property (Theorem \ref{smth.comp.free.rem}), 
%\cite[Theorem 3.3]{bri:eqchow}, 
${\pi_T}^*(u)$ is in the image of 
$\widetilde{i^*_T}$. 
So let 
$y\in A^*_T(\tilde{X})$ be such that $\tilde{i}^*_T(y)=\pi_T^*(u)$. 
To conclude the proof, we need to check that $y$ is in the image of $\pi^*$. 
In view of equivariant Kronecker duality (Proposition \ref{tlinear.kro.prop}), 
this is equivalent to checking that the dual of $y$,   
%via the equivariant Kronecker map $\mathcal{K}_T$, 
namely,  
$\tilde{\varphi}:=\mathcal{K}_T(y)$ is in the image of $\pi_*^t$, the transpose of the 
surjective morphism $\pi_*:A^T_*(\tilde{X})_\Q\to A^T_*(X)_\Q$. Also, we should observe that the functor $\mathcal{K}_T(-)$  
transforms the previous commutative diagram into another one involving the corresponding dual modules ${\rm Hom}(A^T_*(-)_\Q,S_\Q)$. 
Now set $\varphi_u:=\mathcal{K}_T(u)$. By construction, for every codimension-one subtorus $H$, there 
exists $\varphi_u^H:A^T_*(X^H)_\Q\to S_\Q$ such that $\varphi_u=\varphi^H_u \circ {i_{T,H}}_*$. In fact, 
we can place this information into a commutative diagram:  
%But this is done as follows: 
$$
\xymatrix{
A^T_*(\tilde{X^T})_\Q \ar@{->>}[d]^{\pi^T_*}\ar[r]^{\widetilde{i_{T,H}}_*}& A^T_*(\tilde{X}^H)_\Q \ar@{->>}^{\pi_*^H}[d]\ar[r]& A^T_*(\tilde{X})_\Q\ar@{->>}[d]^{\pi_*} \ar[ldd]_{\tilde{\varphi}}\\
A^T_*(X^T)_\Q \ar[r]^{{i_{T,H}}_*}\ar[rd]^{\varphi_u}& A^T_*(X^H)_\Q \ar[r]\ar[d]^{\varphi_u^H}& A^T_*(X)_\Q\ar@{-->}[ld]^{\varphi \; (\exists ?)} \\
           &  S_\Q.         &      \\
}
$$
%where $\tilde{\varphi}=\mathcal{K}_T$
%because $u$ is in the intersection of the images of the various $\pi_H$ (the vertical maps in the middle), and they define corresponding elements at the $\tilde{X}$-level, then 
%there exists $\tilde{\varphi}$ as above. 
In this form, our task reduces to showing that there exists $\varphi$ making the dotted arrow into a solid arrow. 
Bearing this in mind, we claim that  
$\tilde{\varphi}$ is zero on the kernel of $\pi_*$. 
Indeed, let $v\in A^T_*(\tilde{X})_\Q$ be such that $\pi_*(v)=0$. By the localization theorem 
there exists a product of non-trivial characters $\chi_1\cdots \chi_m$ such that 
$\chi_1\cdots \chi_m \cdot v$ is in the image of $\tilde{i}^T_*:A^T_*(\tilde{X}^T)_\Q\to A^T_*(\tilde{X})_\Q$. 
As both of these $S$-modules are free, then, unless $v$ is zero, we have 
$\chi_1\cdots \chi_m \cdot v\neq 0$. Let $w$ be such that 
$\tilde{i}^T_*(w)=\chi_1\cdot \chi_m \cdot v$. 
By commutativity of the diagram, $i^T_*(\pi^T_*(w))=0$. 
But $i^T_*$ 
is injective by Theorem \ref{eqloc.kro.thm}, so $\pi_*^T(w)=0$. Thus  
$$\tilde{\varphi}(\chi_1\cdots \chi_m \cdot v)=\chi_1\cdots \chi_m \cdot \tilde{\varphi}(v)=\varphi_u(\pi_*^T(w))=0.$$ 
As $S_\Q$ has no torsion, we get 
$\tilde{\varphi}(v)=0$, which proves the claim. 
Using this, one easily defines $\varphi$ with the 
sought-after properties. %We are done.  

Finally, for the last assertion of the theorem, recall that group embeddings 
are a special class of spherical varieties known to have resolutions of singularities 
in arbitrary characteristic \cite[Chapter 6]{bk:frob}.
\end{proof}

\begin{rem}   
Theorem \ref{Tlinear.envelopes.thm} and its proof 
can be readily translated into 
the language of 
equivariant operational $K$-theory 
with $\Z$-coefficients. 
See \cite[Section 6]{go:opk} for a 
presentation of the (integral) equivariant 
operational $K$-theory of projective group embeddings. 
%(cf. \cite{go:opk}).   
%coefficients also holds. 
\end{rem}

%Theorem \ref{Tlinear.envelopes.thm} holds for a 
%very large subclass of spherical varieties, namely, 
%projective group embeddings, since they   

%\begin{rem}
%Let $\pi:\tilde{X}\to X$ be a  
%$T$-equivariant envelope and suppose that 
%$\tilde{X}$ is a complete special $T$-scheme. 
%Then $i^*_T:{\rm op}A^*_T(X)\to {\rm op}A^*_T(X^T)$ 
%is also injective (cf. Proposition \ref{inj.fix.set.lem}).   
%\end{rem}

\section{Rational equivariant Chow cohomology of spherical varieties}

Throughout this section we work in characteristic zero. 
The aim is to 
describe the rational equivariant Chow cohomology 
of a spherical variety %(in characteristic zero) 
by comparing it with that 
of an equivariant resolution, using the rational CS property 
(Theorem \ref{Tlinear.envelopes.thm}) and 
equivariant Kronecker duality (Proposition \ref{tlinear.kro.prop}).  
The main result (Theorem \ref{opA_spherical.thm}), 
inspired by \cite[Theorem 7.3]{bri:eqchow}, 
is an extension of Brion's description  
to the setting of equivariant operational Chow groups. 

\medskip

In what follows, we denote by  
$G$ 
a connected reductive 
linear algebraic group  with Borel subgroup 
$B$ and maximal torus $T\subset B$.  
We denote by $W$ the Weyl group of $(G,T)$. 
Observe that $W$ is generated by reflections 
$\{s_\alpha\}_{\alpha \in \Phi}$, 
where $\Phi$ stands for the set of roots of $(G,T)$. 
%As usual, $S=A^*_T$ and 
Recall that $S^W_\Q=({A^*_T}_\Q)^W={A^*_G}\otimes \Q$. 
For the purposes of this section, we shall assume that $G$-spherical varieties are 
{\em locally linearizable} for the induced $T$-action. 

\subsection{Preliminaries}  
Recall that any spherical $G$-variety contains 
only finitely many $G$-orbits;
as a consequence, it contains only finitely many fixed points of $T$. 
Moreover, since ${\rm char}(\k)=0$, any spherical $G$-variety $X$ admits an equivariant 
resolution of singularities, i.e., there exists a smooth $G$-variety $\tilde{X}$ 
together with a proper birational $G$-equivariant morphism $\pi:\tilde{X}\to X$. 
Then the $G$-variety $\tilde{X}$ is also spherical; if moreover $X$ is complete, 
we may arrange so that $\tilde{X}$ is projective. 
Notice that, in general, a resolution of singularities need not be an 
equivariant envelope. The next result gives an important class 
of spherical varieties for which equivariant resolutions 
{\em are} equivariant envelopes. We thank M. Brion for leading us 
to the following proof.

\begin{prop}\label{scs.envelope}
Let $X$ be a normal simply-connected spherical $G$-variety (i.e. the $B$-isotropy 
group of its dense orbit is connected).    
Let $f:\tilde{X}\to X$ be a proper birational morphism. 
Then $f$ is an equivariant envelope.  
\end{prop}

\begin{proof}
%Let $p:\tilde{X}\to X$ be a toroidal resolution of $X$.  
It suffices to show that every $B$-orbit in $X$ is the isomorphic image 
via $f$ of a $B$-orbit in $\tilde{X}$. 
So let $\mathcal{O}=(B)\cdot x=B/B_x$ be an orbit in $X$. 
It follows from \cite{bj:red} that $\mathcal{O}$ has a 
connected isotropy group. The preimage $f^{-1}(\mathcal{O})\subset \tilde{X}$
is of the form $B\times^{B_x} F$, 
where $F$ denotes the fiber $p^{-1}(x)$. 
Since $F$ is connected and complete (by Zariski's main theorem), 
it contains a fixed point $y$ of the connected solvable group $B_x$. 
Then the orbit $B\cdot y$ in $\tilde{X}$ is mapped 
isomorphically to $B\cdot x$.       
\end{proof}

\begin{rem}
Examples of simply-connected spherical varieties include (normal) $G\times G$-equivariant 
embeddings of $G$. 
%In particular, all toric varieties are simply-connected spherical. 
%This fact is used crucially in Payne's proof that the integral 
%equivariant 
%operational Chow ring $A^*_T(Y)$ of {\em any} toric variety $Y$ is 
%isomorphic to %$PP^*_T(Y)_\Z$ 
%a ring of piecewise polynomial functions with $\Z$-coefficients \cite{p:t}.  
%For more general embeddings of reductive groups, 
%the techniques of \cite{p:t} do not quite apply. 
%Instead, GKM theory, as developed in Section 2.8,  
%yields a presentation of the rational Chow cohomology of 
%projective group embeddings. Compare  
\end{rem}

Now we record a few notions and results from \cite[Section 7]{bri:eqchow} needed in our task. 
A subtorus $H\subset T$ 
is called {\em regular} if its 
centralizer $C_G(H)$ is equal to $T$; otherwise 
$H$ is called {\em singular}. A 
subtorus of codimension one is singular if and only if 
it is the kernel of some positive root $\alpha$. 
In this case, 
$\alpha$ is unique and the group $C_G(H)$ is the product 
of $H$ with a subgroup $\Gamma$ isomorphic to $SL_2$ 
or $PSL_2$. Furthermore, if $X$ is a $G$-variety, then $X^H$ 
%the fixed point variety $X^H$ 
inherits an action of 
$C_G(H)/H$, a quotient of $\Gamma$.    

\begin{prop}[\protect{\cite[Proposition 7.1]{bri:eqchow}}]
\label{st.cod.one.fixed.sph.prop} 
 Let $X$ be a spherical $G$-variety. 
Let $H\subset T$ be a subtorus of codimension one. 
Then each irreducible component of $X^{H}$ 
is a spherical $C_G(H)$-variety. 
Moreover,
\begin{enumerate}[(i)]
 \item If $H$ is regular, then $X^{H}$ is at 
most one-dimensional. 
 \item If $H$ is singular, then $X^{H}$ is at most 
two-dimensional. 
If moreover $X$ is complete and smooth, 
then any two-dimensional 
connected component of $X^{H}$ is 
(up to a finite, purely inseparable 
equivariant morphism) either a rational ruled surface 
$
\mathbb{F}_n=\P(\mathcal{O}_{\P^1}\oplus \mathcal{O}_{\P^1}(n)), 
$
where $C_G(H)$ 
acts through the natural action of $SL_2$, or the 
projective plane, where $C_G(H)$ acts through the 
projectivization of a non-trivial $SL_2$-module 
of dimension three. 
\hfill $\square$
\end{enumerate} 
\end{prop}

%Before dealing with the case of singular spherical varieties, 
For later use, we record Brion's presentation of 
the rational equivariant Chow rings of nonsingular   
ruled surfaces. We follow closely the notation and conventions
of \cite[Section 7]{bri:eqchow}. Let 
$D$ be the torus of diagonal matrices in $SL_2$, and let $\alpha$
be the character of $D$ given by 
$$\alpha
\left(
\begin{matrix}
t & 0 \\
0 & t^{-1}
\end{matrix}
\right)=t^2.$$ 
We identify the character ring of $D$ with $\Q[\alpha]$. 
Now consider a rational ruled surface $\F$ with ruling 
$\pi:\F\to \P^1$. Notice that $\F$ has exactly four fixed 
points $x,y,z,t$ of $D$, where $x,y$ (resp. $z,t$) are mapped 
to $0$ (resp. $\infty$) by $\pi$. Moreover, we may assume 
that $x$ and $z$ lie in one $G$-invariant section of $\pi$, 
and that $y$ and $t$ lie in the other $G$-invariant section. 
With this ordering of the fixed points, we identify 
$A^*_T(\F^D)_\Q$ with $\Q[\alpha]^4$. 
In contrast, denote by $\P(V)$ the projectivization 
of a nontrivial $SL_2$-module $V$ of dimension three. 
The weights of $D$ in $V$ are either $-2\alpha$, 0, $2\alpha$ 
(in the case when $V=sl_2$) or $-\alpha$, 0, $\alpha$ 
(in the case when $V=k^2\oplus k)$. We denote by $x,y,z$ the corresponding 
fixed points of $D$ in $\P(V)$, and we identify $A^*_T(\P(V)^D)_\Q$ 
with $\Q[\alpha]^3$.   

\begin{prop}
Notation being as above, the image of 
$$i^*_D:A^*_D(\F)_\Q\to \Q[\alpha]^4$$  
consists of all $(f_x,f_y,f_z,f_t)$ such that 
$f_x\equiv f_y\equiv f_z \equiv f_t \mod \alpha$ 
and $f_x-f_y+f_z-f_t\equiv 0 \mod \alpha^2$. 
On the other hand, 
the image of $$i^*_D:A^*_D(\P(V))_\Q\to \Q[\alpha]^3$$  
consists of all $(f_x,f_y,f_z)$ such 
that $f_x\equiv f_y\equiv f_z\mod \alpha$ 
and $f_x-2f_y+f_z\equiv 0 \mod \alpha^2$. \hfill $\square$
\end{prop}

\subsection{$T$-equivariant Chow cohomology}
Let $X$ be a complete possibly singular spherical $G$-variety. 
Let $H\subset T$ be a singular subtorus of codimension one.  
In order to obtain an explicit description 
of $A^*_T(X)_\Q$ out of Theorem \ref{Tlinear.envelopes.thm}    
%Clearly, %As a second step towards %establishing our main result,  
we need to determine which $C_G(H)$-spherical 
surfaces could appear as irreducible components of $X^H$.     
We will do this by means of 
Proposition \ref{st.cod.one.fixed.sph.prop}. 
(This is the only case of interest to us, for if $H$ is regular, 
then $X^H$ is $T$-skeletal, and GKM theory applies, see Appendix.)  
So let $Y$ be a two-dimensional irreducible component of $X^H$. 
By \cite[Proposition 7.5]{eg:cycles} %Lemma \ref{smooth_envelope.lem} 
we may find a proper 
birational equivariant morphism $X'\to X$ 
with $X'$ smooth and $G$-spherical. 
Thus $Y$ is the image of some irreducible 
component $Y'$ of ${X'}^H$ (by Borel's fixed point theorem). 
Given that $X$ is 
complete, so is $X'$ and, under our considerations,  
$Y'$ is a two-dimensional complete $C_G(H)$-spherical variety. 
Hence $Y'$ is either the projective plane 
or a rational ruled surface (up 
to a finite, purely inseparable equivariant morphism). 
We inspect these two cases in more detail.  

\smallskip

(a) If $Y'=\P^2$, then 
the normalization $\tilde{Y}$ of $Y$ 
is also $\P^2$ (up 
to a finite, purely inseparable equivariant morphism, 
which is, in particular, bijective).  

(b) If $Y'$ is a rational ruled surface, 
then the normalization $\tilde{Y}$ of $Y$ 
is either (i) $Y'$ or (ii) the surface obtained 
by contracting the unique section $\mathcal{C}$
of negative self-intersection in $Y'$ 
(this is a very special weighted projective plane). 
\smallskip

Notice that, 
except for case (b)-(ii), the normalization 
$\tilde{Y}$ of $Y$ is a smooth projective surface 
with finitely many $T$-fixed points.  
In such cases, it readily follows that $A^*_T(\tilde{Y})_\Q$ 
is free of rank $|\tilde{Y}^T|$ (Theorem \ref{smth.comp.free.rem}). 
We show that this property also holds in case (b)-(ii). 
%For this a little computation is needed. 

\begin{lem}\label{w.plane.lem} 
Let 
$P_{n}=\F/\mathcal{C}$ be 
the weighted projective plane obtained 
by contracting the unique section $\mathcal{C}$ 
of negative self-intersection in $\F$.  
Then $A^T_*(P)_\Q$ is a free $S_\Q$-module of rank three. 
Hence, $A^*_T(P)_\Q$ is also $S_\Q$-free of rank three.  
\end{lem}

\begin{proof} 
Clearly, $|P_{n}^T|=3$. The associated BB-decomposition 
of $P_{n}$ consists of three cells: a point, a copy of 
$\A^1$ and an open cell, say $U$, 
isomorphic to $\A^2/\mu_n$, where 
$\mu_n\subset D$ is the cyclic group with eigenvalues 
$(\xi, \xi^{-1})$, where $\xi$ is a $n$-th root of unity. 
Note that $A_*(U)_\Q\simeq A_*(\A^2)^{\mu_n}_\Q$, and 
the latter identifies to $A_*(\A^2)_\Q$, because 
the action of $\mu_n$ on $\A^2$ 
is induced by the action of 
$D$ (a connected group). 
So $A_*(U)_\Q\simeq \Q$. 
This yields the isomorphism 
$A^T_*(U)\simeq S_\Q$ (see e.g. \cite{go:rm}). 
From this, and the fact that the BB-decomposition 
is filtrable, it easily follows 
that the $S_\Q$-module 
$A^T_*(P_{n})_\Q$ is free of rank 3. 
Finally, the second assertion of the lemma follows 
from Proposition \ref{tlinear.kro.prop}.   
\end{proof}

%We now describe the ring $A^*_T(\F/\mathcal{C})_\Q$. 

\begin{cor}\label{w.plane.cor}
Notation being as above, assume that 
$\mathcal{C}$ joins the fixed points $y$ and $t$ of $\F$, so that 
the fixed points of $P_n$ are identified with $x,y,z$. 
Then the image of $i^*_D:A^*_D(P_{n})_\Q\to \Q[\alpha]^3$ 
consists of all $(f_x,f_y,f_z)$ 
such that $f_x\equiv f_y\equiv f_z \mod(\alpha)$ and $f_x-2f_y+f_z\equiv 0 \mod(\alpha^2)$. 
%where $\alpha$ is the character of $D\subset SL_2$ in Proposition 4.4.  
\end{cor}

\begin{proof}
%Without loss of generality, we may assume that $C$ joins the 
%fixed points $y$ and $t$ of $\F$. Now 
Observe that $q:\mathbb{F}_n\to P_n$ is an envelope. 
By Theorem \ref{kimura.thm} and Proposition \ref{inj.fix.set.lem}, 
the problem reduces to find the image of $q^*$. 
%Since restriction to the fixed point set is always injective, 
By Theorem \ref{kimura.thm} again, an element 
$(f_x,f_y,f_z,f_t)\in A^*_T(\F)_\Q$ is in the image of $q^*$ 
if and only if it satisfies the usual relations 
$f_x\equiv f_y\equiv f_z\equiv f_t \mod \alpha$ and 
$f_x-f_y+f_z-f_t\equiv 0 \mod(\alpha^2)$, plus 
the extra relation $f_y=f_t$ 
(which accounts for the fact that $\mathcal{C}$ 
is collapsed to a fixed point in $P_n$). 
Hence, the relation $f_x-f_y+f_z-f_t\equiv 0 \mod(\alpha^2)$ 
reduces to $f_x-2f_y+f_z\equiv 0 \mod(\alpha^2)$, finishing 
the argument.   
\end{proof}

Back to the general setup, let $X$ be a $G$-spherical variety and 
let $H$ be a singular subtorus of codimension one. Let $Y$ 
be an irreducible component of $X^H$, and let 
%$\tilde{Y}$ 
$\pi:\tilde{Y}\to Y$ 
be the normalization map.  
By the previous analysis, 
we know the relations that 
define the image of %the pullback 
$i^*_T:A^*_T(\tilde{Y})_\Q\to A^*_T(\tilde{Y}^T)_\Q$. 
%As a second technical lemma, we will show that 
We claim that $\pi^*:A^*_T(Y)_\Q\to A^*_T(\tilde{Y})_\Q$ 
is in fact an isomorphism. 
First, consider the commutative diagram  
$$
\xymatrix{
A^*_T(Y)_\Q \ar[d]^{\mathcal{K}_T}\ar[r]^{\pi^*}& 
A^*_T(\tilde{Y})_\Q\ar[d]^{\mathcal{K}_T}\\
\Hom_{S_\Q}(A^T_*(Y), S_\Q) \ar[r]^{(\pi_*)^t}& 
\Hom_{S_\Q}(A^*_T(\tilde{Y})_\Q, S_\Q).\\
}
$$
where the vertical maps are isomorphisms 
because of  
equivariant Kronecker duality (Proposition \ref{tlinear.kro.prop}), 
and $(\pi_*)^t$ represents the transpose of the surjective map 
$\pi_*:A^T_*(\tilde{Y})_\Q\to A^T_*(Y)_\Q$ 
(commutativity follows 
from the projection formula). 
Thus, to prove our claim,  
it suffices to show 
that $\pi_*$ is injective.
In fact, since 
$\pi_*$ is a surjective map 
of free $S_\Q$-modules,  
the problem reduces to comparing the ranks 
of $A^T_*(\tilde{Y})_\Q$ and $A^T_*(Y)_\Q$. 
If these ranks agree, we are done, for a surjective 
map of free $S_\Q$-modules of the same rank is 
an isomorphism. Bearing this in mind, 
we invoke the localization theorem 
(Theorem \ref{brion.loc.thm}):  
the ranks of $A^T_*(\tilde{Y})_\Q$ 
and $A^T_*(Y)_\Q$ are $|\tilde{Y}^T|$
and $|Y^T|$ respectively. But 
$|\tilde{Y}^T|=|Y^T|$ by Lemma  
\ref{equality.fix.set.lem}. 
This yields the claim. 

\begin{lem} \label{equality.fix.set.lem}
Let $Y$ be a complete $T$-variety with finitely 
many fixed points. 
Let $p:\tilde{Y}\to Y$ be the normalization. 
If $Y$ is locally linearizable and $\tilde{Y}$ 
is projective, then the normalization 
$p$ induces a bijection $p_T:\tilde{Y}^T\to Y^T$ 
of the fixed point sets. 
\end{lem}
 
\begin{proof}
Clearly, $p$ induces a surjection $p_T:\tilde{Y}^T\to Y^T$. 
Arguing by contradiction, suppose that $p_T$ 
is not injective. Then there are at least two different 
fixed points $x,y\in \tilde{Y}$ such that $p(x)=p(y)$. 
Now choose a $T$-invariant curve $\ell\subset \tilde{Y}$  
passing through $x$ and $y$. It follows that 
the image $\pi(\ell)$ is an 
invariant curve on $Y$ with exactly one fixed point. 
But this is impossible, for the action on 
$Y$ is locally linearizable 
(cf. \cite[Example 4.2]{ti:sph}).  
\end{proof}

With all the ingredients at our disposal, we are now ready 
to state the main result of this section. 
This builds on and extends Brion's result 
(\cite[Theorem 7.3]{bri:eqchow}) 
to the rational 
equivariant Chow cohomology 
of possibly singular complete spherical varieties. 
Our findings complement Brion's deepest results \cite{bri:eqchow}. 

\begin{thm}\label{opA_spherical.thm}
Let $X$ be a complete $G$-spherical variety. 
The image of the injective map 
$$i^*_T:A^*_T(X)_\Q\to A^*_T(X^T)_\Q$$ 
consists of all families 
$(f_x)_{x\in X^T}$ 
satisfying the relations: 
\begin{enumerate}[(i)] 
\item $f_x\cong f_y  \mod \chi$, whenever $x,y$ are connected 
by a $T$-invariant curve with weight $\chi$. 

\item $f_x-2f_y+f_z\equiv 0 \mod \alpha^2$ whenever $\alpha$ 
is a positive root, $x$, $y$, $z$ lie in an 
irreducible component of $X^{\ker{\alpha}}$ 
whose normalization is isomorphic to $\P^2$ or the 
weighted projective plane $P_n$, and $x,y,z$ are 
ordered as in Section 4.1. 

\item $f_x-f_y+f_z-f_t\equiv 0 \mod \alpha^2$ whenever $\alpha$ 
is a positive root, $x$, $y$, $z$, $t$ 
lie in an irreducible component of $X^{\ker{\alpha}}$ 
whose normalization is isomorphic to $\F$, 
and $x,y,z,t$ are 
ordered as in Section 4.1. 
\end{enumerate}
\end{thm}

\begin{proof} 
In light of Theorem \ref{Tlinear.envelopes.thm} and Theorem \ref{gkm.thm},   
%and our 
%analysis of the fixed point loci of codimension one subtori. 
it suffices to consider the case when $H$ 
is a singular codimension one subtorus, i.e. 
$H=\ker{\alpha}$, for some positive root $\alpha$. 
Let $X^H=\bigcup_j X_j$ be the 
decomposition into irreducible components. 
Notice that each $X_j$ is either a fixed point, 
a $T$-invariant curve or a possibly singular 
rational surface (by Proposition 
\ref{st.cod.one.fixed.sph.prop} and 
our previous analysis). 
Now, by Remark \ref{normalization.irr.comp.rem}, 
we have 
the commutative diagram 
$$
\xymatrix{
0\ar[r] & A^*_T(X^{\ker {\pi}})_\Q\ar[r] \ar[d]_{i^*_{T,H}}&
\bigoplus_i A^*_T(X_j)_\Q \ar[r]\ar[d]& 
\bigoplus_{i,j}A^*_T(X_{i,j})_\Q\ar[d]^{\rm id}\\
0\ar[r] & A^*_T(X^T)_\Q\ar[r] &
\bigoplus_i A^*_T(X_j^T)_\Q \ar[r]& 
\bigoplus_{i,j}A^*_T(X_{i,j}^T)_\Q,
}
$$
where each $X_{i,j}$ is at worst a 
union of fixed points and $T$-invariant curves. 
The image of the middle vertical map is 
completely characterized by 
our previous analysis, Lemma \ref{w.plane.lem} and 
Corollary \ref{w.plane.cor}. Hence so 
is the image of $i^*_{T,H}$, as it follows from the conclusion of  
Remark \ref{normalization.irr.comp.rem}.  
Now apply Theorem \ref{Tlinear.envelopes.thm} to conclude the proof. 
\end{proof}

Observe that cases (ii) and (iii) do not occur if  
$X$ is a $G\times G$-equivariant projective 
embedding of $G$, for then %. Indeed, in such case, 
$X$ is $T\times T$-skeletal \cite{go:equiv}. 
In this situation, 
Theorem \ref{opA_spherical.thm} yields  
an explicit description of $A^*_{T\times T}(X)_\Q$; 
this is obtained by appealing to the results of \cite{go:equiv}, 
which identifies all the characters involved in case (i), 
and proceeding as in \cite[Section 6]{go:opk}.  
Moreover, %for group embeddings,  
when $X$ is $\Q$-filtrable \cite{go:rm},  
%a characterization of rational smoothness 
%(in the sense of \cite{go:rm}, a generalization 
%of the topological notion of cohomology manifold 
%to equivariant intersection theory),  
%of rational Chow cohomology, 
the $S_\Q$-module $A^{T\times T}_*(X)_\Q$ is free, 
%for group embeddings, 
and 
there are some criteria for 
Poincar\'e duality in equivariant Chow cohomology. See \cite{go:rm} for details. 
%on possibly singular 
%group embeddings.  
%The results will appear in \cite{go:emb}. 

\subsection{$G$-equivariant Chow cohomology} 
To adapt the definition of 
$G$-equivariant operational Chow groups, in this subsection we work in the 
category of $G$-quasiprojective schemes, i.e.  
$G$-schemes having an ample $G$-linearized invertible sheaf. 
This assumption is fulfilled, e.g., by $G$-stable 
subschemes of normal quasiprojective $G$-schemes \cite{su:eq}. 

\smallskip 
 
The following result is a synthesis of  \cite[Proposition 6]{eg:eqint} 
and \cite[Note 2.5]{vis:char}.

\begin{prop} \label{W.inv.prop}
%Let $G$ be a connected reductive group with maximal torus $T$ 
%and Weyl group $W$. 
For a $G$-scheme $X$,   
%If $X$ is locally linearizable,  
we have $A^G_*(X)_\Q \simeq A^T_*(X)^W_\Q$ 
and $A^*_G(X)_\Q\simeq A^*_T(X)^W_\Q$. \qed   
\end{prop}

%In particular, if $X$ is a projective 
%$G$-spherical variety, one can use Proposition \ref{W.inv.prop} 
%to explicitly compute $A^*_G(X)$ from the description 
%of $A^*_T(X)$ given in Theorem \ref{opA_spherical.thm}.  
%
%\smallskip 

%Let $X$ be a $G$-spherical variety. 
%Proposition \ref{W.inv.prop}, together with 
%Theorem \ref{opA_spherical.thm},   
%yields a description of $A^*_G(X)$ in many situations. 
% $G$-spherical variety $X$,  %a bit further.  
Now we further describe $A^G_*(X)_\Q$ and $A^*_G(X)_\Q$ 
when $X$ is a $G$-spherical variety. %a bit further.    
%and its relation to $A^*_T(X)_\Q$. 
%Our treatment 
%is somewhat terse because the whole issue reduces to the 
%situation discussed in \cite[Appendix]{go:opk}. 

\smallskip

\begin{prop}\label{G.kunneth} 
Let $X$ be a $G$-scheme with finitely many $B$-orbits. 
Then for any $G$-scheme $Y$ the K\"{u}nneth map 
$A^G_*(X)\otimes_{S^W_\Q} A^G_*(Y)\to A^G_*(X\times Y)$ is an isomorphism.  
\end{prop}

\begin{proof} 
Simply argue as in \cite[Theorem A.2]{go:opk}, using 
Proposition \ref{kunneth.prop} and 
the fact that $S_\Q$ is a free $S^W_\Q$-module of rank $|W|$.    
%we obtain the following result. The proof is an easy 
%adaptation of . %, so we omit it. 
\end{proof}

\begin{prop}\label{kro.thm}%%%%%%%%% G-kronecker
Let $X$ be a projective 
$G$-scheme with a finite number of $B$-orbits. 
Then the $G$-equivariant Kronecker map 
$$\mathcal{K}_G:A^*_G(X)_\Q\longrightarrow Hom_{S^W_\Q}(A_*^G(X)_\Q,S^W_\Q) \hspace{1cm} \alpha\mapsto (\beta\mapsto p_{X*}{(\beta\cap \alpha)})$$
is a $S^W_\Q$-linear isomorphism. Moreover, 
we have 
$A^*_T(X)_\Q\simeq A^*_G(X)_\Q\otimes_{S^W_\Q}S_\Q.$  
\end{prop}

\begin{proof}
The first part is formally deduced 
from Proposition \ref{G.kunneth}, 
as in the $T$-equivariant case (Proposition \ref{tlinear.kro.prop}). 
The second one is obtained by adapting 
the proof of \cite[Corollary A.5]{go:opk}. 
\end{proof}

Arguing as in \cite[Corollary 6.7.1]{bri:eqchow}, one obtains 
the next result.

\begin{cor}
Let $X$ be a projective $G$-spherical variety. 
If $A^T_*(X)_\Q$ is $S_\Q$-free, then 
$A^*_G(X)_\Q$ is $S^W_\Q$-free 
and restriction to the fiber induces an 
isomorphism $A^*_G(X)_\Q/{S^W_\Q}_+A^*_G(X)_\Q
\simeq A^*(X)_\Q.$ \qed  
\end{cor}

Corollary 4.12
is satisfied by 
$\Q$-filtrable spherical $G$-varieties \cite{go:rm}.

%%%%%%% REMARK: Observe that a solvable group can act on an irreducible variety 
%%%%%%%%%%%%%%% with an open orbit, but with infinitely many orbits. 
%%%%%%%%%%%%%%% For instance, take the solvable matrix group 
%%%%%%%%%%%%%%%     1 b
%%%%%%%%%%%%%%%     0 a       ,with a\neq 0, 
%%%%%%%%%%%%%%% acting on k^2. 
%%%%%%%%%%%%%%% The whole line spanned by (1,0) is fixed, so there infinitely
%%%%%%%%%%%%%%% many orbits (each point in the line is fixed). 
%%%%%%%%%%%%%%% The open orbit is the orbit of (0,1). 
%%%%%%%%%%%%%%% See more examples in Vinberg's paper on complexity of action of reductive groups.

%\include{note}

%%%%%%%%%%%%%%%%%%%%%%%%%%%%%%%%%%%%%%%%%%%%%%%%%%             %%%%%%%%%%%%%%%%%%%%%%%%%%%%%%%%%%%%%%%%%%%%%%%%%%%%%%%%%%%%%%%
%%%%%%%%%%%%%%%%%%%%%%%%%%%%%%%%%%%%%%%%%%%%%%%%%%   PART II moved  to notion.tex and previous 2.2 version of this paper  %%%%%%%%%%%%%%%%%%%%%%%%%%%%%%%%%%%%%%%%%%%%%%%%%%%%%%%%%%%%%%%
%%%%%%%%%%%%%%%%%%%%%%%%%%%%%%%%%%%%%%%%%%%%%%%%%%             %%%%%%%%%%%%%%%%%%%%%%%%%%%%%%%%%%%%%%%%%%%%%%%%%%%%%%%%%%%%%%%

\section{Further remarks}% and open problems}

\noindent (1) {\em Description of the image of restriction to the fiber 
$i^*:A^*_T(X)\to A^*(X)$ by using equivariant multiplicities}. 
So far, this has been carried out for singular toric varieties \cite{pk:tor}. 
It would be interesting to obtain similar formulas for more general, possibly singular, projective group embeddings. 
%Unlike the case of Chow groups, $i^*$ is in general   
%not surjective and its kernel is not necessarily 
%generated in degree one. See \cite{pk:tor} for an 
%illustration of these claims. 

%\medskip
\smallskip 

\noindent (2) 
{\em Understand the action of $PP^*_T(X)$ on $A^T_*(X)$ 
for $T$-skeletal spherical varieties}, in light of 
Brion's description of the intersection pairing 
between curves and divisors on spherical varieties \cite{bri:cd}. 
This should also provide a geometric interpretation of the 
coefficients arising from the cap and cup product formulas (Corollaries \ref{action.op.chow.on.homology.cor} and \ref{op.cup.prod}). 
This will be 
pursued elsewhere.

\appendix
\setcounter{section}{1}
\section*{Appendix A: Localization theorem and GKM theory for rational 
equivariant Chow cohomology}
Here we translate the 
results of \cite{go:opk} into the language of equivariant Chow cohomology. 
In that paper we studied  
equivariant operational $K$-theory, but as stated in \cite[Section 7]{go:opk} 
the results readily extend to equivariant Chow cohomology with 
rational coefficients. The purpose of this appendix is to supply a detailed proof of  
this claim, merely for the sake of completeness. %This is done here.  
From now on, we assume ${\rm char}(\k)=0$.

%%%%%%%%%%%%%%%   Envelopes and Kimura's work
%\smallskip

%The following was first observed in . 
 \smallskip
 
%Kimura's computation of Chow cohomology implies that ${\rm op}A^*_T(X)$ 
%of a singular 
%variety $X$ injects into ${\rm op}A^*_T(\tilde{X})$ of an 
% equivariant envelope with an explicit cokernel. So if $\tilde{X}$ 
% can be chosen to be nonsingular, then ${\rm op}A^*_T(X)$ injects into the rather geometric ring $A^*_T(\tilde{X})$.    
%More precisely, suppose that 

Let $p:\tilde{X}\to X$ be a $T$-equivariant birational 
envelope which is an 
isomorphism over an open set $U\subset X$. Let $\{Z_i\}$ be the irreducible 
components of $Z=X-U$, and let $E_i=p^{-1}(Z_i)$, with $p_i:E_i\to Z_i$ 
denoting the restriction of $p$. 
The next theorem is Kimura's fundamental result adapted to 
our setup. 
%equivariant operational Chow groups.% due to Kimura. 

\begin{thm}[\protect{\cite[Theorem 3.1]{ki:op}}] \label{kimura.thm}
Let $p:\tilde{X}\to X$ be a $T$-equivariant envelope. 
Then the induced map $p^*:A^*_T(X)\to A^*_T(\tilde{X})$ 
is injective.
Furthermore, if $p$ is birational (and notation is as above), 
then  
the image of $p^*$ is described inductively 
as follows: a class $\tilde{c}\in A^*_T(\tilde{X})$ equals $p^*(c)$, 
for some 
$c\in A^*_T(X)$ if and only if, for all $i$, 
we have $\tilde{c}|_{E_i}=p^*_i(c_i)$ for 
some $c_i\in A^*_T(Z_i)$.  
\hfill $\square$ 
\end{thm}
Since $E_i$ and $Z_i$ have smaller dimension than $X$, we can apply 
this result to
compute $A^*_T(X)$ using a resolution of singularities %(if e.g. ${\rm char}(\k)=0$) 
and induction on dimension.  
In fact, if $\tilde{X}$ is chosen to be smooth, then $A^*_T(\tilde{X})$ 
corresponds to the Chow group of $\tilde{X}$ 
graded by codimension \cite[Proposition 4]{eg:eqint};     
thus $A^*_T(X)\subset A^*_T(\tilde{X})$ sits inside a more geometric object. 
Theorem \ref{kimura.thm} is one of the reasons why Kimura's results \cite{ki:op} make operational Chow groups 
more computable.  %, as we shall see below. 
%(cf. \cite{go:opk}).  

\medskip

%\smallskip

%Now let $G=T$ be an algebraic torus.
In Proposition \ref{inj.fix.set.lem}, 
we state another crucial consequence of Kimura's work. %of equivariant operational Chow groups. 
Put in perspective, it asserts that the rational 
equivariant operational Chow ring $A^*_T(X)_\Q$ of 
{\em any} complete $T$-scheme $X$ is a subring of $A^*_T(X^T)_\Q$. 
Moreover, %for the latter  
there is a natural isomorphism %(over $\Z$)  
$A^*_T(X^T)\simeq A^*(X^T)\otimes_{\Z} S.$     
Indeed, for a fixed degree $j$, \cite[Theorem 2]{eg:eqint}  
yields 
the identifications 
$A^j_T(X^T)\simeq A^j((X^T\times U)/T)\simeq A^j(X^T\times (U/T)),$
where $U$ is an open $T$-invariant subset of a $T$-module $V$, so that 
the quotient $U\to U/T$ exists and is a principal $T$-bundle, and %that 
the codimension of $V\setminus U$ is large enough. %(Theorem \ref{equiv.oper.and.op.equiv.rem}). 
Additionally, %Even more so, 
%as in Example \ref{t.borel.construction.ex}, 
we can find $U$ such that 
$U/T$ is a product of projective spaces (see e.g. \cite{eg:eqint}). It follows that 
$A^*(X^T\times U/T)\simeq A^*(X^T)\otimes A^*(U/T)$, 
by the projective bundle formula \cite[Example 17.5.1 (b)]{f:int}. 
% % 
% %$A^*(X^T)\otimes_{\Q} S$. 
% 
In many cases of interest, $X^T$ is finite (e.g. for spherical varieties) 
and so one has 
$A^*_T(X)_\Q\subseteq \bigoplus^\ell_1 A^*_T(X^T)_\Q=S_{\Q}^{\ell}$, where $\ell=|X^T|$.  
This motivated our introduction of localization techniques, and ultimately 
GKM theory, into the study of rational equivariant operational Chow rings 
and integral equivariant operational $K$-theory \cite{go:opk}.

\begin{prop}\label{inj.fix.set.lem}
Let $X$ be a %projective %(filtrable or projective) 
$T$-scheme. If $X$ is complete, then the pull-back 
$i^*_T:A^*_T(X)_\Q\to A^*_T(X^T)_\Q$ 
is injective.  
\end{prop}

\begin{proof}
The argument is essentially 
that of \cite[Proposition 3.7]{go:opk}. 
We include it  for convenience. 
Choose 
a $T$-equivariant envelope 
$p:\tilde{X}\to X$, with $\tilde{X}$ smooth. %(since ${\rm char}(\k)=0$).   
It follows that 
$p^*:A^*_T(X)\to A^*_T(\tilde{X})$ 
is injective (Theorem \ref{kimura.thm}). 
Since 
$\tilde{X}$ is smooth, 
%the map
%by inclusion of fixed points $\tilde{i}_T:\tilde{X}^T\to \tilde{X}$, 
%namely  
$\tilde{i}^*_T:A^*_T(\tilde{X})_\Q\to A^*_T({\tilde{X}}^T)_\Q$ 
is injective (%by \cite[Proposition 4]{eg:eqint} 
%together with 
by Theorem \ref{smth.comp.free.rem}). 
Now the chain of inclusions 
$\tilde{X}^{T}\subset p^{-1}(X^T)\subset \tilde{X}$  
indicate that $\tilde{i}^*_T$ factors through 
$\iota^*:A^*_T(\tilde{X})\to A^*_T(p^{-1}(X^T))$, 
where 
$\iota:p^{-1}(X^T)\hookrightarrow \tilde{X}$ 
is the natural inclusion. 
Thus, $\iota^*$
is injective over $\Q$ as well. % homomorphism %, namely 
Finally, adding this information to the commutative diagram below
$$
\xymatrix{
A^*_{T}(X)_\Q\ar[r]^{p^*} \ar[d]^{i_T^*}& A^*_{T}(\tilde{X})_\Q\ar[d]^{\iota^*}\\
A^*_{T}(X^T)_\Q \ar[r]^{p^*}& A^*_{T}(p^{-1}(X^T))_\Q. 
}
$$
renders $i^*_T:A^*_T(X)_\Q\to A^*_T(X^T)_\Q$ injective. 
\end{proof}

\begin{cor}[\protect{\cite[Corollary 3.8]{go:opk}}]\label{injectivity.cor}
Let $X$ be a complete $T$-scheme. Let $Y$ be a $T$-invariant 
closed subscheme containing $X^T$. Let $\iota:Y\to X$ be the 
natural inclusion. 
Then the pullback $\iota^*:A^*_T(X)_\Q\to A^*_T(Y)_\Q$ is injective. 
In particular, if $H$ is a closed subgroup of $T$,  
then 
$i_{H}^*:A^*_T(X)_\Q\to A^*_T(X^H)_\Q$ is injective.
%The statement holds over $\Z$ provided 
%$X$ satisfies the additional assumption 
%of Proposition \ref{inj.fix.set.lem}. 
\hfill $\square$
\end{cor}

\begin{rem}\label{normalization.irr.comp.rem}
Let $Y$ be a complete $T$-scheme with irreducible components $Y_1,\ldots, Y_n$. 
Let $Y_{ij}=Y_i\cap Y_j$. 
By Theorem \ref{kimura.thm} %and \cite[Corollary 3.4]{go:opk} 
the following sequence is exact 
$$
0\to A^*_T(Y)_\Q\to \bigoplus_i A^*_T(Y_i)_\Q\to \bigoplus_{i,j}A^*_T(Y_{ij})_\Q.
$$ 
%When $Y^T$ is finite, the sequence above
This sequence yields the 
commutative diagram \cite[Corollary 3.6]{go:opk}:
$$
\xymatrix{
0\ar[r]& A^*_T(Y)_\Q\ar[r]\ar[d]^{i_{T,Y}^*}& \bigoplus_i A^*_T(Y_i)_\Q\ar[r]\ar[d]^{\oplus_i i_{T,Y_i}^*}& \bigoplus_{i,j}A^*_T(Y_{ij})_\Q\ar[d]^{i_{T,Y_{i,j}}^*} \\
0\ar[r]& A^*_T(Y^T)_\Q\ar[r]^{p}& \bigoplus_i A^*_T(Y^T_i)_\Q\ar[r]^{q}& \bigoplus_{i,j}A^*_T(Y^T_{ij})_\Q
}
$$
Since all vertical maps are injective 
(Proposition \ref{inj.fix.set.lem}),  
we can describe the image of the first 
vertical map in terms of the image of the second vertical map 
and the kernel of $q$. Indeed,  
$p({\rm Im}(i^*_{T,Y}))\simeq {\rm Im} (\oplus_i i^*_{T,Y_i})\cap {\rm Ker}(q).$  
Moreover, if $Y^T$ is finite, 
then the kernel of $q$ consists of all families 
$(f_i)_{i=1}^n$ such that 
$f_i(x)=f_j(x)$ %(equality of $k$-components), 
whenever $x\in Y_{ij}^T$. 
\end{rem}

\smallskip

Back to the general case, let $X$ be a complete $T$-scheme.  
One wishes to describe the image of the injective 
map $i^*_T:A^*_T(X)_\Q\to A^T_*(X^T)_\Q.$ 
For this, let $H$ 
be a subtorus of $T$ of codimension one. 
Since   
$i_T$ %X^T\to X$  
factors as $i_{T,H}:X^T\to X^{H}$ 
followed 
by $i_{H}:X^{H}\to X$,  
%That is, 
the image of $i^*_T$ is contained in the image 
of $i^*_{T,H}$. Hence,  
${\rm Im}(i^*_T)%:{\rm op}A^*_T(X)_\Q\to {\rm op}A^*_T(X^T)_\Q]
\subseteqq 
\bigcap_{H\subset T} {\rm Im}(
i^*_{T,H})
%:{\rm op}A^*_T(X^{H})_\Q\to {\rm op}A^*_T(X^T)_\Q
,$
where the intersection runs over all codimension-one subtori $H$ of $T$. %(the same holds over $\Q$). 
This observation yields a complete description 
of the image of $i_T^*$ over $\Q$.  
Before stating it, we recall a definition from  
\cite[Section 4]{go:opk}. 

\begin{dfn}
Let $X$ be a complete $T$-scheme. We say that $X$ has the 
{\em Chang-Skjelbred property} (or {\em CS property}, for short) 
if the pullback  
$i^*_T:A^*_T(X)\to A^*_T(X^T)$  
is injective, and its image 
is exactly the intersection 
of the images of 
$
i^*_{T,H}:A^*_T(X^H)\to A^*_T(X^T),
$
where $H$ runs over all subtori of codimension one in $T$. 
When the defining conditions hold over $\Q$ rather than $\Z$, 
we say that $X$ has the {\em rational CS property}.  
\end{dfn}

Some obstructions for the CS property to hold 
are e.g. (i) $A_*(X^T)$ could have $\Z$-torsion,   
(ii) $\dim{T}\geq 2$ and there exist a $T$-orbit on $X$ 
whose stabilizer is not connected, for instance, if $X$ is nonsingular, $T$-skeletal 
(Definition \ref{genericaction.def})
and the weights of the $T$-invariant curves are not primitive.  

\smallskip

By Theorem \ref{smth.comp.free.rem}, %\cite[Theorem 3.3]{bri:eqchow}, 
every nonsingular 
complete $T$-scheme has the rational CS property. We extend 
this result to include all possibly singular 
complete $T$-schemes. See 
\cite[Theorem 4.4]{go:opk} for the corresponding 
statement in equivariant operational $K$-theory 
with integral coefficients.   

\begin{thm}\label{cs.thm}
If $X$ is a complete $T$-scheme, then it  
%Then $X$ 
has the rational CS property.  
\end{thm}

\begin{proof}
Simply argue as in \cite[Theorem 4.4]{go:opk}, using \cite[Lemma 7.2]{eg:cycles}, Theorem \ref{kimura.thm} 
and Proposition \ref{inj.fix.set.lem}.  
\end{proof}

Before stating our version of GKM theory,
let us recall
a few definitions from \cite{gkm:eqc}, \cite{go:cells} and \cite{go:opk}.

\begin{dfn} \label{genericaction.def} 
Let $X$ be a complete $T$-variety. 
Let $\mu:T\times X\to X$ be the action map.
We say that $\mu$ is a {$T$\em-skeletal action} if 
%\begin{enumerate}
 %\item 
 the number of $T$-fixed points and 
 one-dimensional $T$-orbits in $X$ is finite.
%\end{enumerate}
In this context, $X$ is called a {$T$\em-skeletal variety}. 
The associated graph of fixed points and invariant curves is called 
the {\em GKM graph} of $X$. We shall 
denote this graph by $\Gamma(X)$. 
\end{dfn}

Notice that, in principle, Definition \ref{genericaction.def} 
allows for $T$-invariant irreducible curves with exactly one 
fixed point (i.e. the GKM graph 
$\Gamma(X)$ may have simple loops). 
In \cite[Proposition 5.3]{go:opk} we show that the 
functor $A^*_T(-)_\Q$ 
``contracts'' such loops %of $\Gamma(X)$ 
to a point. The proof there is given in the context 
of operational $K$-theory, but it easily extends 
to our current setup. %of operational Chow groups.  

\begin{prop}[\protect{\cite[Proposition 5.3]{go:opk}}]\label{1.2.fix.point.curves.prop}
Let $X$ be a complete $T$-variety and let $C$ 
be a $T$-invariant irreducible curve of $X$ which is not 
fixed pointwise by $T$. Then the image of the injective map 
$i^*_T:A^*_T(C)_\Q\to A^*_T(C^T)_\Q$ is described as follows: 
\begin{enumerate}[(i)]
 \item If $C$ has only one 
fixed point, say $x$, then $i^*_T:A^*_T(C)_\Q\to A^*_T(x)_\Q$ is 
an isomorphism; that is, $A^*_T(C)_\Q\simeq S_\Q$. 
\smallskip

\item If $C$ has two fixed points, then  
$$A^*_T(C)_\Q\simeq 
\{(f_0,f_\infty) \in S_\Q\oplus S_\Q\,|\, 
f_0\cong f_\infty \mod \chi\},$$ 
where $T$ acts on $C$ via the character $\chi$. \hfill $\square$
\end{enumerate}
\end{prop}

Let $X$ be a complete $T$-skeletal variety. 
It is possible to define a ring $PP_T^*(X)_\Q$ 
of (rational) {\em piecewise polynomial functions}. 
Indeed, let $A^*_T(X^T)_\Q=\bigoplus_{x\in X^T}S_\Q$. 
%where $R_x$ is a copy of the polynomial algebra $S_\Q$. 
We then define
$PP_T^*(X)_\Q$ as the subalgebra of $A^*_T(X^T)_\Q$ defined by
\[
PP_T^*(X)_\Q= \{(f_1,...,f_m)\in\bigoplus_{x\in X^T}S_\Q\;|\; f_i\equiv f_j\;mod(\chi_{i,j})\}
\]
where $x_i$ and $x_j$ are the two (perhaps equal) fixed points 
in the closure of the one-dimensional
$T$-orbit $\mathcal{C}_{i,j}$, and $\chi_{i,j}$ 
is the character of $T$ associated with $\mathcal{C}_{i,j}$.
The character $\chi_{i,j}$ is uniquely determined up to sign (permuting the two
fixed points changes $\chi_{i,j}$ to its opposite). Invariant 
curves with only one fixed point do not impose any relation, 
and this is compatible with Proposition \ref{1.2.fix.point.curves.prop}. 
Now we are ready to state our version of GKM theory. 

\begin{thm}[\protect{\cite[Theorem 5.4]{go:opk}}]\label{gkm.thm}
Let $X$ be a complete $T$-skeletal variety. 
Then the pullback  
%$i^*_T:{\rm op}A^*_T(X)_\Q\longrightarrow {\rm op}A^*_T(X^T)_\Q=\bigoplus_{x_i\in X^T}S_\Q$  
$i_T^*:A^*_T(X)_\Q\to A^*_T(X^T)_\Q$ 
induces an algebra isomorphism between $A^*_T(X)_\Q$ and $PP_T^*(X)_\Q$. \hfill $\square$ 
\end{thm}

%%--------------------Here the manuscript ends--------------------------------


{\footnotesize
\begin{thebibliography}{E-G-99}
\bibitem[B1]{bb:torus}Bialynicki-Birula, A. {\em Some theorems on actions of algebraic groups.} Ann. of Math., 2nd Ser., Vol 98, No. 3, (Nov. 1973),  480-497. 

\bibitem[B2]{bb:decomp}Bialynicki-Birula, A. {\em Some properties of the decompositions of algebraic varieties determined by actions of a torus.} Bull. Acad. Polon. Sci. Ser. Sci. Math. Astronom. Phys. 24 (1976), no. 9, 667-674. 

\bibitem[BCP]{bif:reg}Bifet, E.; De Concini, C.; Procesi,C. {\em Cohomology of regular embeddings}. Adv. in Math., 82 (1990), 1-34.


\bibitem[BJ1]{bj:red} Brion, M., Joshua, R. {\em Intersection cohomology of Reductive Varieties}. J. Eur. Math. Soc. 6 (2004). %, 465-481.


\bibitem[BJ2]{bj:chern} Brion, M., Joshua, R. {\em Equivariant Chow Ring and Chern Classes of Wonderful Symmetric Varieties of Minimal Rank}. Transformation groups, Vol. 13, N. 3-4 (2008), 471-493.

\bibitem[BK]{bk:frob} Brion, M., Kumar, S. {\em Frobenius splitting methods in geometry and representation theory}. Progress in Mathematics, 231. Birkh\"{a}user Boston. 


%\bibitem[Bl]{bloch:higher} Bloch, S. {\em Algebraic cycles and Higher K-theory}. Adv. in Math., 61 (1986), 267-304. 

%\bibitem[Bo]{bo:lag} Borel, A. {\em Linear Algebraic Groups}. Third Edition, Springer-Verlag.
\bibitem[Bo]{bourbaki:mod} Bourbaki, N. {\em Elements of mathematics. Algebra I: Chapters 1-3}. Springer (1989).

%\bibitem[Br1]{bri:sph} Brion, M. {\em Variet\'es sph\'eriques}. Available at http://www-fourier.ujf-grenoble.fr/~mbrion/spheriques.pdf

\bibitem[Br1]{bri:cd} Brion, M. {\em Curves and divisors in spherical varieties}. 
 Algebraic groups and Lie groups, 21-34,
Austral. Math. Soc. Lect. Ser., 9, Cambridge Univ. Press, Cambridge (1997). 

\bibitem[Br2]{bri:eqchow} Brion, M. {\em Equivariant Chow groups for torus actions}. Transf. Groups, vol. 2, No. 3 (1997), 225-267. 

\bibitem[Br3]{bri:ech} Brion, M.       {\em Equivariant cohomology and equivariant intersection theory}. Notes by Alvaro Rittatore. 
NATO Adv. Sci. Inst. Ser. C Math. Phys. Sci., 514,  Representation theories and algebraic geometry (Montreal, PQ, 1997),  1-37, Kluwer Acad. Publ., Dordrecht (1998). 
%\bibitem[Br5]{bri:eu} Brion, M. {\em Poincar\'e duality and equivariant cohomology}. Michigan Math. J. 48, 2000, pp. 77-92.

%\bibitem[Br5]{bri:rat} Brion, M. {\em Rational smoothness and fixed points of torus actions}. Transf. Groups, Vol. 4, No. 2-3, 1999, pp. 127-156.


%\bibitem[BLV]{blv:sph} Brion, M.; Luna, D.; Vust, T. {\em Espaces homog\`enes sph\'eriques}. Invent. Math. 84 (1986), no. 3, 617-632.


\bibitem[C]{c:schu} Carrell, J. {\em The Bruhat graph of a Coxeter group, a conjecture of Deodhar, and rational smoothness of Schubert varieties}. 
%Algebraic groups and their generalizations: classical methods (University Park, PA, 1991), 53–61,
Proc. Sympos. Pure Math., 56, Part 1, Amer. Math. Soc., Providence, RI, 1994. 

\bibitem[Da]{da:tor} Danilov, V. {\em The geometry of toric varieties}. Uspekhi Mat. Nauk 33 (1978), no. 2(200), 85-134, 247. 

\bibitem[DP] {dp:sym} De Concini, C., Procesi, C. {\em Complete symmetric varieties}. Invariant theory (Montecatini, 1982), 1-44,
Lecture Notes in Math., 996 (1983), Springer, Berlin. 

%\bibitem[DP2] {dp:reg} De Concini,C., Procesi, C. {\em Cohomology of compactifications 
%of algebraic groups.} Duke Math. J. 53 (1986), no. 3, 585-594. 

\bibitem[dB] {ba:mot} del Ba\~no, S. {\em On the Chow motive of some moduli spaces}. J. Reine Angew. Math. 532 (2001),
105-132.

\bibitem[EG1]{eg:eqint} Edidin, D., Graham, W. {\em Equivariant Intersection Theory}. Invent. Math. 131 (1998), 595-634. 
\bibitem[EG2]{eg:cycles} Edidin, D., Graham, W. {\em Algebraic cycles and completions of equivariant K-theory}. Duke Math. Jour. Vol. 144, No. 3 (2008).

\bibitem[Ei]{ei:comm} Eisenbud, D. {\em Commutative Algebra with a View Toward Algebraic Geometry}. Springer-Verlag (1995). 
\bibitem[ES]{es:num} Ellingsrud, G., Str\o mme, S. {\em Towards the Chow ring of the Hilbert scheme of $\P^2$}, J. reine angew. Math. 441 (1993), 33-44. 




\bibitem[FMSS]{f:sph} Fulton, W., MacPherson, R., Sottile, F., Sturmfels, B. {\em Intersection theory on spherical varieties}. J. of Alg. Geom., 4 (1994), 181-193.

\bibitem[Fu]{f:int} Fulton, W. {\em Intersection Theory}. Springer-Verlag, Berlin (1984).




\bibitem[G1]{go:cells} Gonzales, R. {\em Rational smoothness, cellular decompositions and GKM theory}. Geometry \& Topology, Vol. 18, No. 1 (2014), 291-326.
\bibitem[G2]{go:equiv} Gonzales, R. {\em Equivariant cohomology of rationally smooth group embeddings}. Submitted. arXiv:1112.1365
\bibitem[G3]{go:opk} Gonzales, R. {\em Localization in equivariant operational $K$-theory and the 
Chang-Skjelbred property}. Submitted. arXiv:1403.4412.

\bibitem[G4]{go:rm} Gonzales, R. {\em Algebraic rational cells, equivariant intersection theory, and Poincar\'e duality}. Submitted. arXiv:1404.2486.

%\bibitem[G5]{go:emb} Gonzales, R. {\em Equivariant intersection theory on group embeddings}. Preprint.  



\bibitem[GKM]{gkm:eqc} Goresky, M., Kottwitz, R., MacPherson, R. {\em Equivariant cohomology, Koszul duality, and the localization theorem.} Invent. Math. 131 (1998), 25-83.

\bibitem[Ja]{jann:kth} Jannsen, U. {\em Mixed motives and algebraic K-theory}. Lecture Notes in Mathematics 1400 (1990), Springer-Verlag, New York.

\bibitem[JK]{jk:chow} Joshua, R.; Krishna, A. {\em Higher K-theory of toric stacks}. 
Ann. Scoula Norm. Sup. Pisa Cl. Sci. To appear. %(arXiv:1210.1057.)
%\bibitem[JK2]{jk:motivic} Joshua, R.; Krishna, A. {\em Motivic cohomology of toric stacks}. In preparation, (2012). 
\bibitem[Jo]{jo:linear} Joshua, R. {\em Algebraic $K$-­theory and higher Chow groups of linear varieties}. Math. Proc. Cambridge Phil. Soc., 130 (2001), 37-60.

\bibitem[Ka]{kar:cell} Karpenko, N.A. {\em Cohomology of relative cellular spaces and of isotropic flag varieties}.
Algebra i Analiz. 12 (2000), 3-69. 
\bibitem[Ki]{ki:op} Kimura, S. {\em Fractional intersection and bivariant theory}. Comm. Algebra 20 (1992), no. 1, 285-302.     
%\bibitem[Ki]{ki:alex} Kimura, S. {\em On the characterization of Alexander schemes}. Compositio Math. 92, 273-284, 1994. 
%\bibitem[Kn1]{knop:sph} Knop, F. {\em The Luna-Vust theory of spherical embeddings}. Proceedings of the Hydebarad Conference 
%on Algebraic Groups (Hydebarad, 1989), (1991), 225-249, Manoj Prakashan, Madras.

%\bibitem[Kn2]{knop:b} Knop, F. {\em On the set of orbits for a Borel subgroup}.  Comment. Math. Helv. 70 (1995), no. 2, 285-309.

\bibitem[KP]{pk:tor}
Katz, E.; Payne, S. 
{\em 
Piecewise polynomials, 
Minkowski weights, and 
localization on toric varieties}. 
Algebra Number Theory 2 (2008), no. 2, 135-155. 

\bibitem[Kr]{kr:chow} Krishna, A. {\em Higher Chow groups of varieties with group action.} 
Algebra Number Theory 7 (2013), no. 2, 449-507.  

\bibitem[LP]{lp:equiv} Littelmann, P.; Procesi, C. 
{\em Equivariant Cohomology of Wonderful Compactifications}. 
Operator algebras, unitary representations, enveloping algebras, and invariant theory, 
Paris (1989). 


%\bibitem[LV]{lv:sph} Luna, D.; Vust, T. {\em Plongements d'espaces 
%homog\`enes}. Comment. Math. Helv. 58 (1983), n.2, 186-245.

\bibitem[P]{p:t} Payne, S. {\em Equivariant Chow cohomology of toric varieties.} Math. Res. Lett. 13 (2006), no. 1, 29-41.     


%\bibitem[R1]{re:lam} Renner, L. {\em Linear algebraic 
%monoids}. Encyclopedia of Math. Sciences, vol. 134 (2005), Springer, Berlin.


%\bibitem[R2]{re:sem} Renner, L. {\em Classification of semisimple varieties}. J. of Algebra, vol. 122, No. 2 (1989), 275-287.


%\bibitem[R3]{re:ratsm} Renner, L. {\em Rationally smooth algebraic monoids}. Semigroup Forum 78 (2009), 384-395.

\bibitem[R]{ro:sph} Rosenlicht, M. {\em Questions of rationality for solvable algebraic groups 
over nonperfect fields}. Annali di Mat. 61 (1963), 97-120. 

\bibitem[Sp]{sp:lag} Springer, T. {\em Linear Algebraic Groups}. Birkh\"{a}user Classics. 2nd Edition.

\bibitem[Su]{su:eq} Sumihiro, H. {\em Equivariant completion}. J. Math. Kyoto Univ. 14 (1974), 1-28.    

%\bibitem[Su2]{su:eq2} Sumihiro, H. {\em Equivariant completion II}. J. Math. Kyoto University, 15(3):573–605, 1975.            

\bibitem[Ti]{ti:sph} Timashev, D. {\em Homogeneous spaces and equivariant embeddings}. Encyc. Math. Sci. 138 (2011), Springer-Verlag. 

\bibitem[To]{to:linear} Totaro, B. {\em Chow groups, Chow cohomology and linear varieties}. Forum of Mathematics, Sigma (2014), Vol. 2, e17. 1-25. 
(First appeared in 1996 in preprint form.)  
 

%\bibitem[U]{uma:kth} Uma, V. {\em Equivariant $K$-theory of compactifications of algebraic groups.} Transf. Groups 12 (2007), No. 2, 371-406. 

\bibitem[VV]{vv:hkth} Vezzosi, G., Vistoli, A. {\em Higher algebraic $K$-theory for actions of diagonalizable algebraic groups.}  
Invent. Math. 153 (2003), No. 1, 1-44. 



%\bibitem[Vin]{vin:comp} Vinberg, E. B. {\em Complexity of action of reductive groups}.  Funktsional Anali. i Prilozhen. 20 (1986), no. 1, 1-13, 96.

%\bibitem[Vis]{vis:alex} Vistoli, A. {\em Alexander duality in intersection theory}. Compositio Math., tome 70, n.3 (1989), p. 199-225.

\bibitem[Vi]{vis:char} Vistoli, A. {\em Characteristic classes of principal bundles in algebraic intersection theory}. Duke Math. J. Vol. 58 (1989), n. 2, 299-315
\end{thebibliography}
}
\end{document}